\documentclass[a4paper,10pt]{amsart}
\usepackage[margin=1in]{geometry}
\usepackage{enumerate, amsmath, amsfonts, amssymb, amsthm, mathtools, thmtools, wasysym, graphics, graphicx, xcolor, frcursive,xparse,comment,ytableau,stmaryrd,bbm,array,colortbl,tensor,arydshln,leftidx, mathrsfs}
\usepackage[all]{xy}
\usepackage[boxed,norelsize]{algorithm2e}
\usepackage{hyperref}
\usepackage{cancel}
\usepackage{calligra}
\usepackage[vcentermath]{youngtab}
\ytableausetup{centertableaux}

\usepackage{bm}
\DeclareMathAlphabet{\mathcalligra}{T1}{calligra}{m}{n}

\makeatletter
\newcommand{\specialcell}[1]{\ifmeasuring@#1\else\omit$\displaystyle#1$\ignorespaces\fi}
\makeatother

\usepackage{url, hypcap}
\hypersetup{colorlinks=true, citecolor=darkblue, linkcolor=darkblue}
\definecolor{join}{RGB}{0,77,178}
\usepackage{tikz}
\usetikzlibrary{calc,through,backgrounds,shapes,matrix,cd}

\definecolor{darkblue}{rgb}{0.0,0,0.7} 
\newcommand{\darkblue}{\color{darkblue}} 
\definecolor{darkred}{rgb}{0.7,0,0} 
 \definecolor{lightgrey}{rgb}{0.7,0.7,0.7}

\definecolor{meet}{RGB}{255,205,111}
\definecolor{join}{RGB}{0,77,178}

\DeclareFontEncoding{OT2}{}{}

\newtheorem{theorem}{Theorem}[section]
\newtheorem{proposition}[theorem]{Proposition}
\newtheorem{corollary}[theorem]{Corollary}
\newtheorem{lemma}[theorem]{Lemma}

\theoremstyle{definition}
\newtheorem{definition}[theorem]{Definition}
\newtheorem{example}[theorem]{Example}

\newtheorem{conjecture}[theorem]{Conjecture}
\newtheorem{assumption}[theorem]{Assumption}
\newtheorem{question}[theorem]{Question}

\newtheorem{remarkx}[theorem]{Remark}
\newenvironment{remark}
  {\pushQED{\qed}\remarkx}
  {\popQED\endremarkx}

\usepackage[procnames]{listings}
\usepackage[nameinlink]{cleveref}
\usepackage[all]{xy}
\usepackage[T1]{fontenc}
\Crefformat{footnote}{#2\footnotemark[#1]#3}
\Crefname{conjecture}{Conjecture}{Conjectures}
\Crefname{assumption}{Assumption}{Assumptions}
\Crefname{subsection}{Subsection}{Subsections}
\Crefname{remarkx}{Remark}{Remarks}

\usepackage[cal=boondoxo]{mathalfa}
\usepackage[colorinlistoftodos]{todonotes}
\newcommand{\defn}[1]{\emph{\darkblue #1}}

\numberwithin{equation}{subsection}

\usetikzlibrary{math}
\usepackage{xifthen}
\usepackage{xstring}
\SetKwInput{kwset}{set}
\usetikzlibrary{arrows,backgrounds,calc,trees}
\pgfdeclarelayer{background}
\pgfsetlayers{background,main}
\renewcommand{\mod}{\operatorname{mod}}

\newcommand{\U}{\mathcal{U}}

\newcommand{\Fix}{\mathrm{Fix}}

\newcommand{\GL}{\mathrm{GL}}
\renewcommand{\U}{\mathrm{U}}
\newcommand{\Feg}{\operatorname{Feg}}

\newcommand\blfootnote[1]{%
  \begingroup
  \renewcommand\thefootnote{}\footnote{#1}%
  \addtocounter{footnote}{-1}%
  \endgroup
}

\DeclareMathAlphabet{\mathcal}{OMS}{cmsy}{m}{n}

\allowdisplaybreaks

\title[Rational Catalan Numbers]{Rational Catalan Numbers for Complex Reflection Groups}

\author[W.~Miller]{Weston Miller}
\address[W.~Miller]{University of Texas at Dallas}
\email{Weston.Miller@utdallas.edu}

\begin{document}
\begin{abstract}
    Assuming standard conjectures, we show that the canonical symmetrizing trace evaluated at powers of a Coxeter element produces rational Catalan numbers for irreducible spetsial complex reflection groups. This extends a technique used by Galashin, Lam, Trinh, and Williams to uniformly prove the enumeration of their noncrossing Catalan objects for finite Coxeter groups. 
\end{abstract}

\maketitle

\section{Introduction}
\subsection{Catalan combinatorics}
\blfootnote{This work was supported in part by the National Science Foundation under grant DMS-2246877}
The prototypical noncrossing Coxeter-Catalan objects are the \defn{noncrossing partitions}. In type $A$ with the usual Coxeter element $c=(1,2,\ldots,n)$, these correspond to partitions $\{B_1,\dots,B_k\}$ of the set $\{1,2,\dots,n\}$ such that there do not exist $a < b < c < d$ such that $a,c \in B_i$ and $b,d \in B_j$ with $i \neq j$.  These type $A$ noncrossing partitions are counted by $\operatorname{Cat}_n:=\frac{1}{n+1}\binom{2n}{n}$.

More generally, let $W$ be a finite Coxeter group acting in the reflection representation $V$, and let $c$ be a Coxeter element of $W$.  The \defn{$c$-noncrossing partition lattice} $\mathrm{NC}(W,c)$ is the interval $[e,c]_T$ in the absolute order, and the number of noncrossing partitions is counted by the \defn{Coxeter-Catalan number}
\[
    \operatorname{Cat}(W) := \prod_{i=1}^n \frac{h +1 + e_i}{d_i},
\]
where $h = d_n$ is the Coxeter number of $W$, $d_1 \leq \cdots \leq d_n$ are the degrees of a set of algebraically independent homogeneous polynomials which generate the algebra of invariants $\operatorname{Sym}(V^*)^W$, and $e_i = d_i - 1$.

In type $A$, the \defn{rational Catalan numbers} $\operatorname{Cat}_{a,b}:=\frac{1}{a+b}\binom{a+b}{a}$ count \defn{rational Dyck paths}---lattice paths in an $a \times b$ rectangle (for $a$ and $b$ coprime) that stay above the diagonal. For $a = n$ and $b = n+1$, this recovers the usual Catalan numbers.  In other types, these rational Dyck paths are generalized to certain \defn{nonnesting objects}, but such objects are only defined for Weyl groups.
\begin{figure}[h!]
\begin{center}
    \begin{tikzpicture}[scale = 0.4]
        \draw[thin] (0,0) grid (3,5);
        \draw[thin] (0,0) -- (3,5);
        \draw[ultra thick,black] (0,0) -- (0,1) |- (0,2) |- (0,3) |- (0,4) |- (0,5) |- (1,5) |- (2,5) |- (3,5);
        \draw (5,0) grid (8,5);
        \draw{(5,0) -- (8,5)};
        \draw[ultra thick,black] (5,0) -- (5,1) |- (5,2) |- (5,3) |- (5,4) |- (6,4) |- (6,5) |- (7,5) |- (8,5);
        \draw (10,0) grid (13,5);
        \draw{(10,0) -- (13,5)};
        \draw[ultra thick,black] (10,0) -- (10,1) |- (10,2) |- (10,3) |- (10,4) |- (11,4) |- (12,4) |- (12,5) |- (13,5);
        \draw (15,0) grid (18,5);
        \draw{(15,0) -- (18,5)};
        \draw[ultra thick,black] (15,0) -- (15,1) |- (15,2) |- (15,3) |- (16,3) |- (16,4) |- (16,5) |- (17,5) |- (18,5);
        \draw (20,0) grid (23,5);
        \draw{(20,0) -- (23,5)};
        \draw[ultra thick,black] (20,0) -- (20,1) |- (20,2) |- (20,3) |- (21,3) |- (21,4) |- (22,4) |- (22,5) |- (23,5);
        \draw (25,0) grid (28,5);
        \draw{(25,0) -- (28,5)};
        \draw[ultra thick,black] (25,0) -- (25,1) |- (25,2) |- (26,2) |- (26,3) |- (26,4) |- (26,5) |- (27,5) |- (28,5);
        \draw (30,0) grid (33,5);
        \draw{(30,0) -- (33,5)};
        \draw[ultra thick,black] (30,0) -- (30,1) |- (30,2) |- (31,2) |- (31,3) |- (31,4) |- (32,4) |- (32,5) |- (33,5);
    \end{tikzpicture}
\end{center}
\caption{The rational Dyck paths counted by $\operatorname{Cat}_{3,5}$}
\end{figure}
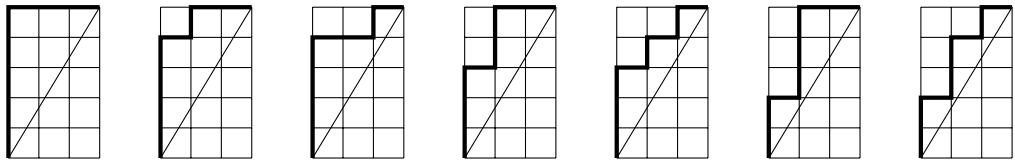

In \cite{galashin2022rational}, the authors define rational noncrossing Coxeter-Catalan objects called \defn{maximal $\mathbf{c}^p$-Deograms}, counted by the \defn{rational Coxeter-Catalan numbers}
\[
    \operatorname{Cat}_p(W) := \prod_{i=1}^n \frac{p + (pe_i \mod h)}{d_i},
\]
where $p$ is coprime to $h$.  These objects are defined for finite Coxeter groups.  As part of their type-uniform proof of this enumeration, they use Hecke algebra traces to compute the point count of braid Richardson varieties over a finite field, producing $q$-deformed rational Catalan numbers:
\[
     \operatorname{Cat}_p(W; q) := \prod_{i=1}^n \frac{[p + (pe_i \mod h)]_q}{[d_i]_q}.
\]

Many of the objects used in their proof can also be defined for the well-generated complex reflection groups, so it is natural to try to compute these traces in the complex case. It turns out that the {\it well-generated} condition is too weak for certain representation-theoretic techniques to work---the necessary condition is for the group to be {\it spetsial}. These spetsial complex reflection groups (see \Cref{def:spets}) are well-generated complex reflection groups that behave as if they were the Weyl group for some connected reductive algebraic group. In the Shephard-Todd classification, the \defn{spetsial} groups are 
\[
    W(A_{n-1}), \quad G(m,1,n), \quad G(m,m,n), \quad G_i \text{ where } i \in \{4,6,8,14,23,\dots,30,32,\dots,37\}.    
\]
Analogs of unipotent characters, Harish-Chandra theory, and Lusztig’s Fourier transform can be defined combinatorially for these groups, which allows techniques from the representation theory of finite groups of Lie type to be extended to spetsial complex reflection groups. 

As the main result of this paper, \Cref{thm:main}, we show that for irreducible spetsial complex reflection groups the trace of a power of a Coxeter element still produces a rational Catalan number even though there are not braid Richardson varieties in this context. Precisely, we prove the following result:
\begin{theorem}
    Let $W$ be an irreducible spetsial complex reflection group with Coxeter number $h$, and let $c$ be a $\zeta_h$-regular element of $W$. Let $\mathbf{c} \in B(W)$ be a lift of $c$ such that $\mathbf{c}^h = \bm{\pi}$. Then
    \[
        \tau_q(T_{\mathbf{c}}^{-p}) = q^{-np}(1-q)^n \operatorname{Cat}_p(W; q).
    \]
\end{theorem}

\subsection{Parking combinatorics}

The \defn{noncrossing parking functions} in type $A$, called 2-partitions in \cite{edelman1980chain}, are sets of tuples $\{(B_1,L_1),\dots,(B_k,L_k)\}$, where $B_i, L_i \subseteq \{1,\dots,n\}$ and
\begin{itemize}
    \item $\{B_1,\dots,B_k\}$ is a noncrossing partition of $\{1,\dots,n\}$,

    \item $\{L_1,\dots,L_k\}$ is a set partition of $\{1,\dots,n\}$, and 

    \item $|B_i| = |L_i|$ for $i = 1,\dots,k$.
\end{itemize}
The number of these noncrossing parking functions is $(n+1)^{n-1}$.

More generally, for $W$ a finite Coxeter group, let $\mathcal{L}$ be the lattice of \defn{flats}. That is, the elements of $\mathcal{L}$ are the intersections of collections of reflection hyperplanes, and they are ordered by reverse inclusion. Then by \cite{brady2002partial} there is an embedding $\mathrm{NC}(W,c) \hookrightarrow \mathcal{L}$. Define an equivalence relation on $W \times \mathcal{L}$ by setting $(w,X) \sim (w',X')$ if $X = X'$ and $wW_X = w'W_X$, where $W_X$ is the pointwise stabilizer of $X$. Denote by $[w,X]$ the equivalence class of $(w,X)$. Then the \defn{$W$-noncrossing parking functions} defined in \cite{armstrong2015parking} are
\[
    \operatorname{Park}^{\mathrm{NC}}_W := \{[w,X] : w \in W \text{ and } X \in \mathrm{NC}(W,c)\},
\]
and they are counted by $(h+1)^n$.

In type $A$, the \defn{rational nonnesting parking functions} for $a$ and $b$ coprime are rational Dyck paths in an $a \times b$ rectangle with a labelling of the north steps with the numbers $\{1,\dots,a\}$, such that the labels increase in each column going north. These parking functions are counted by $b^{a-1}$ \cite[Corollary 4]{armstrong2016rational}.  As with rational Catalan objects, nonnesting parking objects again are only defined for Weyl groups.
\begin{figure}[h!]
    \begin{center}
        \begin{tikzpicture}[scale = 0.5]
            \draw[thin] (0,0) grid (7,5);
            \draw[thin, dotted] (0,0) -- (7,5);
            \draw[ultra thick] (0,0) -- (0,1) |- (0,2) |- (1,2) |- (1,3) |- (1,4) |- (2,4) |- (3,4) |- (4,4) |- (4,5) |- (5,5) |- (6,5) |- (7,5);
            \node at (0.5,0.5) {2};
            \node at (0.5,1.5) {4};
            \node at (1.5,2.5) {1};
            \node at (1.5,3.5) {5};
            \node at (4.5,4.5) {3};
        \end{tikzpicture}
    \end{center}
    \caption{A $(5,7)$-parking function}
\end{figure}
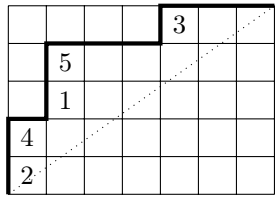

In \cite{galashin2022rational}, the authors uniformly defined rational noncrossing parking objects for finite Coxeter groups. These parking objects correspond to certain walks in the Hasse digram of the weak Bruhat order and are counted by $p^n$. As a corollary of our main result, we prove the following (see \Cref{cor:parking}):
\begin{corollary}
    For $W$ an irreducible spetsial complex reflection group, let $\mathcal{B}$ be a basis of the spetsial Hecke algebra $\mathcal{H}_q(W)$ (adapted to the Wedderburn decomposition), and let $\mathbf{c}$ be a lift of a $\zeta_h$-regular element such that $\mathbf{c}^h = \mathbf{\pi}$. Then
    \[
        \sum_{b \in \mathcal{B}} \tau_q(b^\vee T_{\mathbf{c}^p} b) = (q-1)^n [p]_q^n.
    \]
\end{corollary} 
In the real case, this is a key algebraic step in the proof of the enumeration of rational noncrossing parking functions \cite[Corollary 6.15]{galashin2022rational}. Finding a combinatorial interpretation of the left-hand-side of this equation, e.g.\ rational noncrossing parking functions for spetsial complex reflection groups, is an open problem.

\subsection{Outline of the paper}
In \Cref{sec:ComplexRefl} and \Cref{sec:braidHecke} we cover the necessary background on complex reflection groups, braid groups, and Hecke algebras. A key result in these sections is Tits' deformation theorem (\Cref{thm:tits}) which links the representation theory of a complex reflection group with that of its Hecke algebra. We also show in \Cref{thm:galDeg} that the exterior powers of Galois twists of the reflection representation are precisely where the generic degrees fail to vanish when evaluated at a root of unity. Some data associated to these representations is presented throughout the paper as a running example.

Then, in \Cref{sec:FGOLT} we give some brief background on the representation theory of Lie type to motivate the constructions in \Cref{sec:spetses} and \Cref{sec:LFT}. In particular, we introduce the uniform almost characters, principal series characters, and families of unipotent characters.

The construction of unipotent characters and generic degrees for the infinite families are then described in \Cref{sec:spetses}. We also highlight the role of Clifford theory in describing the representation theory of $G(m,m,n)$. There is a focus throughout this section on computing data for the exterior powers of Galois twists. In particular, the proof of \Cref{thm:galDeg} for the infinite families appears in this section.

The key lemma for our main result is the focus of \Cref{sec:LFT}. We describe Lusztig's original construction of a Fourier transform for Weyl groups, then discuss generalizations to finite Coxeter groups and spetsial complex reflection groups. Finally, in \Cref{sec:ratioCatalan} we present our main result and a corollary related to parking combinatorics.

For the exceptional groups, some of our results are checked by computer \cite{Code}.

\subsection*{Acknowledgments} The author would like to thank Nathan Williams for his guidance, support, and many helpful discussions.

\section{Complex reflection groups}
\label{sec:ComplexRefl}

Most of the background material in this section can be found in~\cite{lehrer2009unitary}. Let $V$ be a finite-dimensional complex vector space. A linear transformation $g \in \GL(V)$ is a \defn{reflection} if the order of $g$ is finite and the subspace $\Fix(g) := \{v \in V : gv = v\}$ has codimension 1. In this case, $\Fix(g)$ will be called the \defn{reflection hyperplane} of $g$. 

\begin{definition}
A \defn{(finite) complex reflection group} is a finite subgroup of $\GL(V)$ that is generated by reflections. 
\end{definition}

For a complex reflection group $W$, we will use $\mathcal{R}$ to denote the set of reflections in $W$, and $\mathcal{A}$ will denote the corresponding set of reflecting hyperplanes (for Coxeter groups, $|\mathcal{R}| = |\mathcal{A}|$).

For any finite subgroup $G$ of $\GL(V)$, there exists a $G$-invariant inner product on $V$. So we may assume that a complex reflection group $G \subset \GL(V)$ is a subgroup of the unitary group $\U(V)$ for an appropriate inner product. Moreover, one can show that finite subgroups of $\U(V)$ are conjugate in $\U(V)$ if and only if they are conjugate in $\GL(V)$.

\begin{definition}
    A complex reflection group $W \subset \U(V)$ is an \defn{irreducible complex reflection group} if $V$ is an irreducible $W$-module.
\end{definition}

For an orbit of hyperplanes $\mathcal{C} \in \mathcal{A}/W$, we will let $e_\mathcal{C}$ denote the order of the pointwise stabilizer $W_H = \{w \in W : wh = h, \,\, \forall h \in H\}$ for any $H \in \mathcal{C}$ (the order does not depend on the choice of $H$). For any $H \in \mathcal{C}$, the group $W_H$ is cyclic with order $e_{\mathcal{C}}$, and there is a reflection $s_H \in W$ with reflecting hyperplane $H$ and determinant $\zeta_{e_\mathcal{C}} = \exp(2\pi i / e_{\mathcal{C}})$. Such reflections are called \defn{distinguished reflections}.

The \defn{field of definition} $k_W$ of a complex reflection group $W$ is the field generated by the traces of the elements of $W$ on the reflection representation. The field of definition is a subfield of $\mathbb{R}$ when $W$ is a finite Coxeter group and equals $\mathbb{Q}$ when $W$ is a Weyl group. It is a theorem of \cite{benard1976schur} and \cite{bessis1997corps} that every representation of $W$ is definable over its field of definition.

\subsection{Classification of complex reflection groups}

Let $\mu_m$ denote the cyclic subgroup of $\mathbb{C}^\times$ consisting of the $m$th roots of unity. If $G$ is a group acting on the set $\{1,2,\dots,n\}$ and $V$ is an $n$-dimensional complex inner product space with orthonormal basis $e_1,\dots,e_n$, the \defn{standard monomial representation} of $\mu_m \wr G$ is given by
\[
    (h,g)e_i = h_{g(i)}e_{g(i)},
\]
where $h = (h_1,\dots,h_n) \in \mu_m^n$.

For $p$ a divisor of $m$, let $A(m,p,n) = \{(h_1,\dots,h_n) \in \mu_m^n : (h_1\cdots h_n)^{m/p} = 1$\}. Let $G(m,p,n) = A(m,p,n) \rtimes S_n$. The group $G(m,p,n)$ is a subgroup of $\mu_m \wr S_n$, so it may be represented as a group of linear transformations in the standard monomial representation.

\begin{definition}
    A complex reflection group $W \subset U(V)$ is \defn{imprimitive} if for some $m > 1$, there is a direct sum $V = V_1 \oplus V_2 \oplus \cdots \oplus V_m$ of non-zero subspaces $V_i$ such that the action of $W$ on $V$ permutes the subspaces $V_1$, $V_2$, \dots, $V_m$. If $W$ is not imprimitive, then it is \defn{primitive}.
\end{definition}

\begin{proposition}[\cite{lehrer2009unitary}, Proposition 2.10]
    The groups $G(m,p,n)$ with $n > 1$ are imprimitive complex reflection groups. $G(m,p,n)$ is irreducible except for the groups $G(1,1,n)$ and $G(2,2,2)$
\end{proposition} 

\begin{proposition}[\cite{lehrer2009unitary}, Proposition 2.14]
    If $V$ is an $n$-dimensional complex inner product space and $G$ is an irreducible imprimitive subgroup of $U(V)$, which is generated by reflections, then $n > 1$ and $G$ is conjugate to $G(m,p,n)$ for some $m > 1$ and some divisor $p$ of $m$.
\end{proposition}

The group $G(m,p,1)$ is the cyclic group of order $m/p$ and is an irreducible primitive complex reflection group. The group $G(1,1,n)$ is the symmetric group $S_n$. Let $W(A_{n-1})$ be the complex reflection group gotten by restricting the action of $G(1,1,n)$ to the hyperplane orthogonal to $e_1 + \cdots + e_n$. Then $W(A_{n-1})$ is irreducible, and it is primitive when $n \geq 5$. 

By the classification due to Shephard and Todd, an irreducible complex reflection group is conjugate in $\U(V)$ to a group $G(m,p,n)$ or one of 34 exceptional (primitive) groups (see \cite[\S 8.7]{lehrer2009unitary}), which are numbered $G_4,G_5,\dots,G_{37}$.

\subsection{Degrees and exponents}

For a complex reflection group $W$ with reflection representation $V$ of dimension $n$, let $S := \text{Sym}(V^*) \simeq \mathbb{C}[x_1,\dots,x_n]$ be the symmetric algebra of $V^*$. Then $W$ acts on $S$ via $(g\cdot p)(v) = p(g^{-1}v)$ for $g \in W$, $p \in S$, and $v \in V$. By the Shephard-Todd-Chevalley theorem (see \cite{chevalley1955invariants}) the \defn{algebra of invariants} $S^W$ is generated by a collection of algebraically independent homogeneous polynomials. The \defn{degrees} of $W$ are defined to be the degrees $d_1 \leq \cdots \leq d_n$ of these generators. The Poincar\'{e} polynomial $P_W$ is defined by 
\[
    P_W := \prod_{j=1}^n [d_j]_q,
\]
where $[n]_q$ denotes the $q$-analog $[n]_q := (q^n - 1)/(q-1)$ for $n \in \mathbb{Z}$.

Let $S^W_+$ be the ideal of $S$ generated by the positive degree elements of $S^W$. Then the action of $W$ on the \defn{coinvariant algebra} $S / S^W_+$ is the regular representation, so $S / S^W_+$ contains exactly $r$ copies of any irreducible representation $M$ of $W$ of dimension $r$ (see \cite{lehrer2009unitary} Section 3.6). The \defn{exponents} of $M$ are defined to be the degrees $e_1(M) \leq \cdots \leq e_r(M)$ of the homogeneous components of $S/S^W_+$ containing a copy of $M$. If $\chi$ is the irreducible character corresponding to $M$, we will also denote $e_i(\chi) := e_i(M)$.

\begin{definition}
    The degrees of $W$ satisfy $d_i = e_i + 1$ for $i = 1,\dots,n$, where $e_i := e_i(V)$. One can define the \defn{codegrees} $d_1^* \geq \cdots \geq d_n^*$ of $W$ by $d_i^* = e_{n-i+1}(V^*) - 1$. An irreducible complex reflection group $W$ is \defn{well-generated} if the degrees and codegrees satisfy $d_i + d_i^* = d_n$. Equivalently, $W$ is well-generated if and only if it can be generated by $n$ reflections.
\end{definition}

It is a theorem of Shephard and Todd (see \cite[Theorem 4.14]{lehrer1999reflection}) that the the degrees and exponents satisfy $|W| = d_1\cdots d_n$ and $|\mathcal{R}| = e_1 + \cdots + e_n$.

\begin{definition}
    Let $W$ be an irreducible complex reflection group. The \defn{Coxeter number} of $W$ is 
    \[
        h := \frac{|\mathcal{R}| + |\mathcal{A}|}{n}.
    \]
    If $W$ is well-generated, then $h = d_n$.
\end{definition}

More generally, following \cite{gordon2012catalan} and \cite{douvropoulos2018enumerating}, define the \defn{generalized Coxeter number} $h_\chi$ associated to a character $\chi$ to be the normalized trace of the central element $\sum_{r \in \mathcal{R}}(1-r)$. That is,
\[
    h_\chi = \frac{1}{\chi(1)}\sum_{r \in \mathcal{R}}(\chi(1) - \chi(r)) = |\mathcal{R}| - \frac{1}{\chi(1)}\sum_{r \in \mathcal{R}} \chi(r).
\]
Equivalently (see \cite[Remark 4.15]{douvropoulos2018enumerating}),
\[
    h_{\chi} = \frac{N(\chi) + N(\chi^*)}{\chi(1)},
\]
where $N(\chi)$ is the sum of the exponents of $\chi^*$ (see \cite[\S 4.B]{broue1997}). These generalized Coxeter numbers are integers by \cite[Corollary 4.16]{douvropoulos2018enumerating}, and $h_\phi = h$ when $\phi$ is the character of the reflection representation. In \cite{trinh2021hecke} and \cite{galashin2022rational}, the authors define the \defn{content} $c(\chi)$ of $\chi$ as
\[
    c(\chi) := \frac{1}{\chi(1)}\sum_{r \in \mathcal{R}} \chi(r) = |\mathcal{R}| - h_{\chi}.
\]

\begin{remark}
    For the symmetric groups, the content of a character $\chi$ is equal to the content of the partition of the integer partition $(\lambda_1 \geq \lambda_2 \geq \cdots \geq 0)$ corresponding to $\chi$. That is, 
    \[
        c(\chi) = \sum_{i=1}^\infty \sum_{j=1}^{\lambda_i} (j-i).
    \]
\end{remark}

\begin{example}[Galois twists]
\label{ex:galCox}
    Let $W$ be a well-generated irreducible complex reflection group with Coxeter number $h$ and reflection representation $V$ of dimension $n$. For well-generated groups, we have the containment $k_W \subseteq \mathbb{Q}(\zeta_h)$. Every element $\sigma \in \operatorname{Gal}(\mathbb{Q}(\zeta_h) / \mathbb{Q})$ acts as $\sigma_p: \zeta_h \mapsto \zeta_h^p$ for some $p$ coprime to $h$.  The \defn{Galois twist} $V^{\sigma_p}$ of $V$ is the irreducible representation of $W$ obtained by applying $\sigma_p$ to the matrices representing the elements $w \in W$ as linear operators on $V$. These Galois twists may not be distinct---different choices of $p$ may result in the same representation.  

    If $r \in \mathcal{R}$, then it has a unique eigenvalue $\zeta \neq 1$ (which is a root of unity) and all other eigenvalues are 1. One can check case-by-case that the order of $\zeta$ divides $h$ in the exceptional groups. In the infinite family $G(m,p,n)$, either $\zeta$ is an $(m/p)$th root of unity (in which case the order of $\zeta$ divides $h = \max\{(n-1)m, mn/p\}$) or $\zeta = -1$ (see \cite[Lemma 2.8]{lehrer1999reflection}). Let $\chi_{k,p}$ denote the character of the irreducible representation $\Lambda^k V^{\sigma_p}$ for $p$ coprime to $h$ and $k \in \{0,1,\dots,n\}$. Then
    \[
        \chi_{k,p}(r) = e_k(\bar{\zeta}, 1, \dots, 1) = \binom{n-1}{k-1} \bar{\zeta} + \binom{n-1}{k},
    \]
    where
    \[
        e_k(x_1,\dots,x_n) = \sum_{i_1 < \cdots < i_k} x_{i_1}\cdots x_{i_k}
    \]
    is the $k$th elementary symmetric polynomial on $n$ variables and 
    \[
        \bar{\zeta} = \sigma_p(\zeta) = \begin{cases}
            \zeta^p & \text{if } \zeta^h = 1, \\
            -1 & \text{if } \zeta = -1.
        \end{cases}
    \]
    Let $h_{k,p}$ be the generalized Coxeter number of $\chi_{k,p}$. Hence
    \begin{align*}
        h_{k,p} &= |\mathcal{R}| - \frac{1}{\binom{n}{k}} \sum_{\mathcal{C} \in \mathcal{A}/W} |\mathcal{C}| \sum_{j=1}^{e_\mathcal{C} - 1} \left(\binom{n-1}{k-1} \bar{\zeta}_{e_{\mathcal{C}}}^{j} + \binom{n-1}{k} \right) \\
        &= |\mathcal{R}| - \frac{1}{\binom{n}{k}} \sum_{\mathcal{C} \in \mathcal{A}/W} |\mathcal{C}| \left(- \binom{n-1}{k-1} + (e_{\mathcal{C}} - 1)\binom{n-1}{k} \right) \\
        &= |\mathcal{R}| - \frac{1}{\binom{n}{k}}\left(-\binom{n-1}{k-1}|\mathcal{A}| + \binom{n-1}{k}|\mathcal{R}| \right) \\
        &= |\mathcal{R}| + \frac{k}{n}|\mathcal{A}| - \left(1 - \frac{k}{n}\right)|\mathcal{R}| \\
        &= kh.
    \end{align*}
\end{example}

\begin{definition}
    A vector $v \in V$ is \defn{regular} if it is not contained in any reflection hyperplane. Let $V^{\text{reg}} := V \setminus \bigcup_{H \in \mathcal{A}} H$ denote the set of regular vectors. An element $c \in W$ is \defn{regular} if it has a regular eigenvector. Moreover, $c$ is \defn{$\zeta$-regular} if this eigenvector may be chosen to have eigenvalue $\zeta$. In this case, the multiplicative order $d$ of $\zeta$ is a \defn{regular number} for $W$.
\end{definition}

The following theorem was first proven case-by-case \cite[Theorem C]{lehrer1999reflection}, but was later proven uniformly \cite[Theorem 3.1]{lehrer2003invariant}.

\begin{theorem}
    A positive integer $d$ is a regular number of $W$ if and only if it divides the same number of degrees as codegrees.
\end{theorem}

If $W$ is well-generated, then the Coxeter number $h$ is a regular number because it divides only $d_n$ and $d_1^*$. It follows that there exists a $\zeta$-regular element for every $h$-th root of unity $\zeta$ (see \cite[Remark 11.23]{lehrer2009unitary}). For $\zeta$ a primitive $h$-th root of unity, the $\zeta$-regular elements of $W$ are called \defn{Coxeter elements}.

\begin{remark}
    Some authors define the Coxeter elements to be only the $\zeta_h$-regular elements. In several cases, we will find it useful to restrict our attention to this subset of the Coxeter elements. It follows easily from the definition that each Coxeter element is a power of some $\zeta_h$-regular element.
\end{remark}

\begin{definition}
    The \defn{fake degree} $\Feg_\chi(q)$ of an irreducible character $\chi$ of $W$ is the graded multiplicity of the irreducible representation with character $\chi$ in $S/S^W_+$:
    \[
        \Feg_\chi (q) = \sum_{i = 1}^r q^{e_i(\chi)}.
    \]
\end{definition}

\begin{remark}
\label{rem:fegDef}
    The reader should be aware that some authors use the symmetric algebra $\operatorname{Sym}(V)$ rather than $\operatorname{Sym}(V^*)$ to define fake degrees $\overline{\operatorname{Feg}}_{\chi}(q)$. In particular, this alternative definition is used in GAP3. These two definitions are related by
    \[
        \operatorname{Feg}_\chi (q) = \overline{\operatorname{Feg}}_{\chi^*}(q).
    \]
    In the real case, the two definitions coincide.
\end{remark}

\begin{example}[Galois twists]
\label{ex:galoisFeg}
    Let $W$ be a well-generated irreducible complex reflection group with reflection representation $V$ and Coxeter number $h$. By \cite[Theorem 2.13, Corollary 3.2]{orlik1980unitary}, the fake degree of $\Lambda^k V^{\sigma_p}$ is given by
    \[
        \sum_{i_1 < \cdots < i_k} q^{e_{i_1}(V^{\sigma_p}) + \cdots + e_{i_k}(V^{\sigma_p})}.
    \]
    
    Moreover, the sets $\{e_1(V^{\sigma_p}), \dots, e_n(V^{\sigma_p})\}$ and $\{pe_1 \mod h,\dots,pe_n \mod h\}$ coincide: It follows easily from \cite[Proposition 4.5]{springer1974regular} that the sets $\{e_1(V^{\sigma_p}), \dots, e_n(V^{\sigma_p})\}$ and $\{pe_1,\dots,pe_n\}$ coincide mod $h$; this is \cite[Proposition 8.1.2]{stump2020cataland}. It then suffices to show that $e_i(V^{\sigma_p}) < h$. For the exceptional groups, this can be checked by computer \cite{Code}. We will prove it for the families $G(m,1,n)$ and $G(m,m,n)$ later (see \Cref{ex:galFeg1} and \Cref{ex:galFeg2}).
\end{example}

\begin{question}
    Is there a uniform proof that $e_i(V^{\sigma_p}) < h$?
\end{question}

\section{Braid groups and Hecke algebras}
\label{sec:braidHecke}

\subsection{Braid groups of complex reflection groups}

By \cite[Corollary 1.6]{steinberg1964differential}, the action of $W$ on $V$ is free precisely on $V^\text{reg}$. It is then not hard to show that the quotient map $\rho : V^\text{reg} \to V^\text{reg}/W$ is a Galois covering. 

\begin{definition}
    The \defn{pure braid group} for a complex reflection group $W$ is $P(W) := \pi_1(V^\text{reg})$.  Its \defn{braid group} is $B(W) := \pi_1(V^\text{reg} / W)$.
\end{definition}
The quotient $\rho$ induces a surjection $\pi : B(W) \to W$, giving a short exact sequence
\[
    1 \to P(W) \xrightarrow[]{\rho_*} B(W) \xrightarrow[]{\pi} W \to 1
,\]
where $W$ can be interpreted as the group of deck transformations of the covering (see \cite[Proposition 1.40]{hatcher_2001}).

The braid group $B(W)$ has a set of generators $\{\mathbf{s}_{H,\gamma}\}$ called \defn{generators of the monodromy} or \defn{braid reflections} (see \cite[Section 2.B]{broue1998complex} or \cite[Section 4.2.5]{broue2010introduction}), such that $\pi(\mathbf{s}_{H,\gamma}) = s_H$ is a distinguished reflection. Moreover, the pure braid group $P(W)$ is generated by the $\{\mathbf{s}_{H,\gamma}^{e_{\mathcal{C}}}\}$ (where $H \in \mathcal{C}$), and so $W \cong B(W) / \langle \mathbf{s}_{H,\gamma}^{e_{\mathcal{C}}} \rangle$.

\begin{definition}
    Define the \defn{full twist} $\bm{\pi} \in P(W)$ by 
    \begin{align*}
        [0,1] &\to V^{\text{reg}} \\
        t &\mapsto v \exp(2\pi i t),
    \end{align*}
    where the basepoint $v \in V^{\text{reg}}$ is suppressed. The image $\rho_*(\bm{\pi}) \in B(W)$ is a central element of $B(W)$ and will also be called the full twist and be denoted by $\bm{\pi}$.
\end{definition}

\begin{proposition}[\cite{broue2010introduction}, Proposition 5.24] \label{thm:rootOfTwist}
    If $g$ is a $\zeta$-regular element of $W$, where $\zeta = \exp(2\pi i m/d)$ is a primitive $d$th root of unity, then $g$ has a lift $\mathbf{g} \in B(W)$ such that $\mathbf{g}^d = \bm{\pi}^m$.
\end{proposition}

\subsection{Hecke algebras for complex reflection groups}

Define variables $\mathbf{u} := (u_{\mathcal{C},j})_{(\mathcal{C} \in \mathcal{A}/W),(0 \leq j \leq e_{\mathcal{C}} - 1)}$, and let $\mathbb{Z}[\mathbf{u}^{\pm 1}]$ be the ring of Laurent polynomials in these variables.

\begin{definition}
    Let $J$ be the ideal of the group algebra $\mathbb{Z}[\mathbf{u}^\pm] B(W)$ generated by the elements
    \[
        (\mathbf{s}_{H} - u_{\mathcal{C},0})(\mathbf{s}_{H} - u_{\mathcal{C}, 1}) \cdots (\mathbf{s}_{H} - u_{\mathcal{C}, e_{\mathcal{C}} - 1}),
    \]
    where $\mathcal{C} \in \mathcal{A} / W$, $H \in \mathcal{C}$, and $\mathbf{s}_{H}$ is a braid reflection (since generators $\mathbf{s}_{H,\gamma}$ and $\mathbf{s}_{H,\gamma'}$ of the monodromy around $H$ are conjugate in $B(W)$ by \cite[Lemma 2.13]{broue1998complex}, it suffices to use only one such braid reflection for each $H$ in the above relations to generate $J$). The \defn{generic Hecke algebra} $\mathcal{H}(W)$ is the quotient $\mathbb{Z}[\mathbf{u}^\pm]B(W) / J$. For $\mathbf{g} \in B(W)$, we'll denote by $T_{\mathbf{g}}$ the corresponding element in the Hecke algebra.
\end{definition}

\begin{theorem}[\textbf{The BMR freeness theorem}]
    The algebra $\mathcal{H}(W)$ is a free $\mathbb{Z}[\mathbf{u}^\pm]$-module of rank $|W|$.
\end{theorem}

A ring homomorphism $\mathbb{Z}[\mathbf{u}^\pm] \to R$ induces a specialization $\mathcal{H}(W) \otimes_{\mathbb{Z}[\mathbf{u}^\pm]} R$ of the Hecke algebra. For example, the specialization $u_{\mathcal{C},j} \xmapsto[]{\sigma} \zeta_{e_\mathcal{C}}^j$ gives the group algebra of $W$ (so that the Hecke algebra may be thought of as a deformation of the group algebra of $W$). A specialization of the Hecke algebra is \defn{admissible} if it factors through this specialization.

\begin{definition}
    The \defn{spetsial Hecke algebra} $\mathcal{H}_q(W)$ is the admissible cyclotomic Hecke algebra induced by the map
    \[
        \theta_q : u_{\mathcal{C},j} \mapsto \begin{cases}
            q & \text{if } j = 0 \\
            \zeta_{e_{\mathcal{C}}}^j & \text{if } j > 0.
        \end{cases}
    \]
    This is a generalization of the 1-parameter Iwahori-Hecke algebra of Coxeter groups.
\end{definition}

\subsection{The representation theory of Hecke algebras}

Most of the material in this section is from the more general representation theory of symmetric algebras and can be found in \cite[Section 7]{geck2000characters}.

If $H$ is an associative $R$-algebra (where $R$ is a commutative ring), then a \defn{trace function} on $H$ is an $R$-linear map $\tau : H \to R$ such that $\tau(hh') = \tau(h'h)$ for all $h,h' \in H$. A trace function $\tau$ is a \defn{symmetrizing trace} if the bilinear form 
\[
    H \times H \to R \qquad (h,h') \mapsto \tau(hh')
\]
is non-degenerate. We say that $H$ is a \defn{symmetric algebra} if it possesses a symmetrizing trace.

\begin{example}
\label{ex:trace}
    There is a canonical symmetrizing trace on the group algebra $\mathbb{C}W$ given by $\tau(w) = \delta_{1 w}$. If $W$ is a real reflection group, then $\mathcal{H}(W)$ has an analogous canonical symmetrizing trace given by $\tau(T_w) = \delta_{1w}$ (see \cite[Proposition 8.1.1]{geck2000characters}).
\end{example}

For non-real complex reflection groups, there is not an obvious choice of canonical lift $T_w$ for each $w \in W$, so there is not a simple generalization to these groups of the construction in \Cref{ex:trace}. It is conjectured that $\mathcal{H}(W)$ has a canonical symmetrizing trace for all complex reflection groups $W$. This is the BMM symmetrizing trace conjecture:

\begin{assumption}[\cite{broue1999towards}, Assumption 2.1(1)]
\label{assum:symTrace}
    There exists a $\mathbb{Z}[\mathbf{u}^{\pm}]$-linear map $\tau : \mathcal{H}(W) \to \mathbb{Z}[\mathbf{u}^\pm]$ such that:
    \begin{enumerate}
        \item $\tau$ is a symmetrizing trace.

        \item Through the specialization $u_{\mathcal{C},j} \mapsto \zeta_{e_{\mathcal{C}}}^j$, the form $\tau$ becomes the canonical symmetrizing trace on the group algebra.

        \item For all $\mathbf{b} \in B(W)$, 
        \[
            \tau(T_{\mathbf{b}^{-1}})^\vee = \frac{\tau(T_{\mathbf{b}\bm{\pi}})}{\tau(T_{\bm{\pi}})},
        \]
        where $\alpha \mapsto \alpha^\vee$ is the automorphism on $\mathbb{Z}[\mathbf{u}^\pm]$ consisting of simultaneous inversion of the indeterminates.
    \end{enumerate}
\end{assumption}

By \cite[Proposition 2.2]{broue1999towards}, if such a symmetrizing trace exists, it is unique. We will call $\tau$ the canonical symmetrizing trace on $\mathcal{H}(W)$. The BMM symmetrizing trace conjecture has been proved for the infinite family $G(m,p,n)$ (see \cite{bremke1997reduced} and \cite{malle1998symmetric}), but remains open for most of the exceptional groups.

The next result is essentially \cite[Theorem 7.2.6]{geck2000characters}, but the precise formulation presented here follows \cite[Proposition 7.12]{broue1999towards} so as to be more easily applied to Hecke algebras. Let $R$ be an integrally closed commutative noetherian domain with field of fractions $F$, and let $H$ be an $R$-algebra which is a free $R$-module of finite rank. Assume that $H$ is endowed with a symmetrizing trace $\tau$. For each (absolutely) irreducible character $\chi$ of $FH := F \otimes_R H$, let $\omega_\chi : Z(H) \to R$ denote the central character associated with $V_\chi$ (i.e., the restriction of the natural map $FH \to \operatorname{End}_F(V_\chi)$), where $V_\chi$ is the $FH$-module corresponding to $\chi$. The \defn{Schur element of $\chi$} is the element of $R$ defined by $S_\chi := \omega_\chi(\chi^\vee)$, where $\chi^\vee$ is the unique element of $H$ satisfying $\tau(\chi^\vee h) = \chi(h)$ for all $h \in H$. 

\begin{theorem}
     If $FH$ is split semisimple. Then 
     \[
        \tau = \sum_{\chi \in \operatorname{Irr}(FH)} \frac{1}{S_\chi} \chi.
     \]
\end{theorem}

Before applying this result to Hecke algebras and their admissible specializations, we first discuss Tits' deformation theorem \cite[Theorem 7.4.6]{geck2000characters}, which gives a bijection between the irreducible representations of an admissible specialization of the Hecke algebra and those of the group algebra.

\begin{theorem}[\cite{malle1999fake}, Theorem 5.2]
\label{thm:splittingField}
    Let $\mu(k_W)$ be the group of roots of unity in $k_W$, and define variables $\mathbf{v} = (v_{\mathcal{C},j})_{(\mathcal{C} \in \mathcal{A}/W),(0 \leq j \leq e_{\mathcal{C}}-1)}$ that satisfy
    \[
        v_{\mathcal{C},j}^{|\mu(k_W)|} = \zeta_{e_{\mathcal{C}}}^{-j} u_{\mathcal{C},j}.
    \]
    Then the field $K_W := k_W(\mathbf{v})$ is a splitting field for $\mathcal{H}(W)$. We will define $\mathcal{H}_{K_W}(W) := \mathcal{H}(W) \otimes_{\mathbb{Z}[\mathbf{u}^\pm]} K_W$.
\end{theorem}

We may extend the specialization $\sigma$ of $\mathcal{H}(W)$ (giving the group algebra) by $v_{\mathcal{C},j} \xmapsto[]{\sigma} 1$. This induces a map $\tilde{\sigma} : \mathcal{H}_{K_W}(W) \to k_W[W]$ which agrees with the surjection $\pi : B(W) \to W$. We are now in the proper setting to apply Tits' deformation theorem (the precise statement here follows \cite[Theorem 4.6]{douvropoulos2018enumerating}).
\begin{theorem}[Tits' deformation theorem]
\label{thm:tits}
    The algebra $\mathcal{H}_{K_W}(W)$ is semisimple and the specialization $\sigma$ induces a bijection
    \[
        d_{\sigma} : \operatorname{Irr}(\mathcal{H}_{K_W}(W)) \to \operatorname{Irr}(k_W[W])
    \]
    that respects the spectra of elements. That is, for any irreducible $\mathcal{H}_{K_W}(W)$-module $U$ the following diagram commutes:
    \begin{center}
        \begin{tikzcd}
            \mathcal{H}_{K_W}(W) \arrow[r, "p_U"] \arrow[d, "\tilde{\sigma}"] & k_W[\mathbf{v}^{\pm}][x] \arrow[d, "t_\sigma"] \\
            k_W[W] \arrow[r, "p_{d_\sigma(U)}"] & k_W[x]
        \end{tikzcd}
    \end{center}
    where the horizontal maps send an element to its characteristic polynomial, and the vertical maps are naturally induced by $\sigma$. Let $\chi_{\mathbf{v}}$ and $\chi$ be the characters associated to $U$ and $d_\sigma(U)$, respectively. Then we have
    \[
        \chi(g) = \sigma(\chi_{\mathbf{v}}(T_{\mathbf{g}})),
    \]
    for all $\mathbf{g} \in B(W)$ where $g = \pi(\mathbf{g})$.
\end{theorem}
Tits' deformation theorem can be applied to any admissible specialization of $\mathcal{H}(W)$ by first moving to the splitting field determined by \Cref{thm:splittingField}. In particular, for the spetsial Hecke algebra $\mathcal{H}_q(W)$, the splitting field is $k_W(y)$ with $y^{|\mu(k_W)|} = q$. We will denote by $\chi_q$ the irreducible character of $\mathcal{H}_q(W)$ corresponding to the irreducible character $\chi$ of $W$.

The next result is a slight reformulation of \cite[Proposition 4.18]{broue1997} following \cite[Proposition 4.11]{douvropoulos2018enumerating}
\begin{proposition}
    Suppose that $W$ is an irreducible complex reflection group. If $\mathbf{w} \in B(W)$ is a lift of $w \in W$ such that $\mathbf{w}^d = \bm{\pi}^\ell$, then
    \[
        \chi_q(T_{\mathbf{w}}) = \chi(w) \cdot q^{(nh - h_\chi) \ell /d},
    \]
    for $\chi$ an irreducible character of $W$.
\end{proposition}

\begin{corollary}
\label{cor:charRootOfTwist}
    Suppose $W$ is an irreducible well-generated complex reflection group with Coxeter number $h$, and $c$ is a $\zeta_h$-regular element of $W$. Let $\mathbf{c} \in B(W)$ be a lift of $c$ such that $\mathbf{c}^h = \bm{\pi}$ (which exists by \Cref{thm:rootOfTwist}). Then
    \[
        \chi_q(T_{\mathbf{c}}^{-p}) = q^{(h_{\chi} - nh) p/ h}\operatorname{Feg}_{\chi}(e^{2\pi i  p/ h}),
    \]
    for $\chi$ an irreducible character of $W$ and $p$ an integer.
\end{corollary}
\begin{proof}
    This follows from the previous result and \cite[Proposition 4.5]{springer1974regular}.
\end{proof}

Returning to Schur elements, by \cite[Proposition 2.5]{broue1999towards}, there exist $S_\chi(\mathbf{v}) \in \mathbb{Z}_{k_W}[\mathbf{v}^{\pm}]$ for $\chi \in \operatorname{Irr}(W)$, such that
\[
    \tau = \sum_{\chi \in \operatorname{Irr}(W)} \frac{1}{S_{\chi}(\mathbf{v})} \chi_{\mathbf{v}},
\]
where $\tau$ is the canonical symmetrizing trace on $\mathcal{H}(W)$ and $\mathbb{Z}_{k_W}$ is the ring of integers of $k_W$. These Schur elements have been computed even in the cases for which \Cref{assum:symTrace} is still open (see \cite{malle2000generic} for instance). The specialization of the canonical symmetrizing trace on $\mathcal{H}(W)$ is a symmetrizing trace $\tau_q$ on $\mathcal{H}_q(W)$ whose Schur elements, which will be denoted $S_{\chi}(q)$, are the images of the Schur elements of $\mathcal{H}(W)$ via the specialization $\theta_q$.

We now define the class of spetisal complex reflection groups.
\begin{definition}
\label{def:spets}
    An complex reflection group is called \defn{spetsial} if all of its irreducible components $W$ satisfy any of the following equivalent conditions (equivalence of the conditions is shown in \cite[Proposition 8.1]{malle2000generic}):
    \begin{enumerate}
        \item $S_{\bm{1}}(q) = P_W$, where $\bm{1}$ denotes the trivial representation of $W$
        \item $P_W / S_{\chi}(q) \in k_W(q)$ for all $\chi \in \operatorname{Irr}(W)$
        \item $P_W / S_{\chi}(q) \in k_W[y]$ for all $\chi \in \operatorname{Irr}(W)$, where $y^{|\mu(k_W)|} = q$
        \item $W$ is one of the following groups:
        \[
            W(A_{n-1}), \quad G(m,1,n), \quad G(m,m,n), \quad G_i \text{ where } i \in \{4,6,8,14,23,\dots,30,32,\dots,37\}.
        \]
    \end{enumerate}
\end{definition}

\begin{definition}
\label{def:genDeg}
    For $W$ an irreducible spetsial complex reflection group, define the \defn{generic degree} of an irreducible character $\chi$ by
    \[
        \operatorname{Deg}_\chi (q) := P_W / S_{\chi}(q) \in k_W(q).
    \]
    Let $a_\chi$ be the largest non-negative integer such that $q^{a_\chi}$ divides $\operatorname{Deg}_\chi(q)$, and let $A_\chi$ be the degree of $q$ in $\operatorname{Deg}_\chi(q)$. Then $h_\chi = a_\chi + A_\chi$ \cite[Corollary 6.9]{broue1999towards}.
\end{definition}

Similarly, define $b_\chi$ to be the smallest integer $i \geq 0$ such that $\chi$ occurs in the character of the $i$th symmetric power of the reflection representation of $W$. That is, $b_\chi$ is the largest non-negative integer such that $q^{b_\chi}$ divides the fake degree of $\chi^*$. Also, let $B_\chi$ be the degree in $q$ of the fake degree of $\chi^*$.

A representation $\psi \in \operatorname{Irr}(W)$ is \defn{special} if $a(\psi) = b(\psi)$. If $W$ is a spetsial group, then for all $\chi \in \operatorname{Irr}(W)$, we have $a(\chi) \leq b(\chi)$ and there exists a special $\psi \in \operatorname{Irr}(W)$ with $a(\chi) = a(\psi)$ \cite[Proposition 8.1]{malle2000generic}.

\begin{proposition}
\label{thm:traceRootOfTwist}
    Suppose $W$ is an irreducible spetsial complex reflection group with Coxeter number $h$, and let $c$ be a $\zeta_h$-regular element of $W$. Let $\mathbf{c} \in B(W)$ be a lift of $c$ such that $\mathbf{c}^h = \bm{\pi}$. Then
    \[
        \tau_q(T_{\mathbf{c}}^{-p}) = \frac{1}{P_W} \sum_{\chi \in \operatorname{Irr}(W)} q^{(h_\chi - nh)p/h} \operatorname{Feg}_\chi(e^{2\pi i p/h}) \operatorname{Deg}_\chi(q),
    \]
    for $\chi$ an irreducible character of $W$.
\end{proposition}
\begin{proof}
    It is straightforward to check case-by-case that an irreducible spetsial complex reflection group is well-generated. The result then follows from \Cref{cor:charRootOfTwist}.
\end{proof}

\begin{theorem}
\label{thm:galDeg}
    Suppose $W$ is an irreducible well-generated complex reflection group with Coxeter number $h$. Then for $p$ relatively prime to $h$,
    \[
        [S_{\bm 1} / S_{\chi}] (e^{2 \pi i p/ h}) = \begin{cases}
            (-1)^k & \text{if } \chi = \chi_{k,p}, \\
            0 & \text{otherwise},
        \end{cases}
    \]
    for $\chi$ an irreducible character of $W$ and $\chi_{k,p}$ the character of $\Lambda^k V^{\sigma_p}$. In particular, if $W$ is spetsial, then
    \[
        \operatorname{Deg}_\chi (e^{2 \pi i p/ h}) = \begin{cases}
            (-1)^k & \text{if } \chi = \chi_{k,p}, \\
            0 & \text{otherwise}.
        \end{cases}
    \]
\end{theorem}
\begin{proof}
    For the symmetric groups (in fact, all real groups), the proof is sketched in \cite[remark 6.10]{galashin2022rational}. For the exceptional irreducible well-generated groups, this is checked by computer \cite{Code}. The proof for the groups $G(m,1,n)$ and $G(m,m,n)$ follows the proof of the ``untwisted case'' in \cite[Proposition 4]{michel2022tower} and will be given in the next section.
\end{proof}

\section{Finite groups of Lie type}
\label{sec:FGOLT}

This section serves primarily to motivate the constructions in the subsequent sections by explaining what is going on for Weyl groups. Most of the material in this section can be found in \cite{carter1985finite}, \cite{digne2020representations}, and \cite{geck2020character}. For background on (linear) algebraic groups, the standard references are \cite{borel2012linear}, \cite{humphreys2012linear}, and \cite{springer1998linear}. See also \cite{geck2013introduction} for a slightly more elementary introduction. Our conventions and notation will usually follow \cite{geck2020character}.

We fix the following objects for this section: Let $W$ be an irreducible Weyl group with generating set of simple reflections $S$, and let $q$ be a power of a prime $p$. Let $\mathbf{G}$ be a simple connected reductive group over $\overline{\mathbb{F}}_p$ with connected center which has Weyl group $W$, and let $F : \mathbf{G} \to \mathbf{G}$ be a Frobenius map with respect to some $\mathbb{F}_q$-rational structure which acts trivially on $W$. We denote by $\mathbf{G}^F$ the corresponding finite group of Lie type and fix a maximally split torus $\mathbf{T}_0$.

\subsection{Uniform almost characters and principal series characters}

A character $\rho \in \operatorname{Irr}(\mathbf{G}^F)$ is called a \defn{unipotent character} if $\langle R_{\mathbf{T}}^{\mathbf{G}}(1_{\mathbf{T}}), \rho \rangle \neq 0$ for some $F$-stable maximal torus $\mathbf{T} \subseteq \mathbf{G}$. Here $1_{\mathbf{T}}$ is the trivial character of $\mathbf{T}^F$ and $R_{\mathbf{T}}^{\mathbf{G}}(1_{\mathbf{T}})$ is the induced Deligne-Lusztig character of $\mathbf{G}^F$. We denote by $\operatorname{Uch}(\mathbf{G}^F)$ the set of all unipotent characters of $\mathbf{G}^F$. We now discuss two important subsets of the unipotent characters of $\mathbf{G}^F$ whose degrees are equal to the fake degrees and generic degrees, respectively, of corresponding characters of $W$ evaluated at $q$.

Any maximal torus of $\mathbf{G}$ has the form ${}^g\mathbf{T}_0 := g \mathbf{T}_0 g^{-1}$ for some $g \in \mathbf{G}$. Moreover, ${}^g\mathbf{T}_0$ is $F$-stable if and only if $g^{-1}F(g) \in N_{\mathbf{G}}(\mathbf{T}_0)$, and the map ${}^g\mathbf{T}_0 \mapsto \pi(g^{-1}F(g))$ determines a bijection between the $\mathbf{G}^F$-classes of $F$-stable maximal tori of $G$ and the conjugacy classes of $W$. Here $\pi : N_{\mathbf{G}}(\mathbf{T}_0) \to W$ is the quotient map. If $w \in W$ is conjugate to $\pi(g^{-1}F(g))$, we say that $\mathbf{T}_w := {}^g\mathbf{T}_0$ is obtained from $\mathbf{T}_0$ by twisting with $w$. We denote by $R_w$ the Deligne-Lusztig character $R_{\mathbf{T}_w}^\mathbf{G}(1_{\mathbf{T}_w^F})$.

For each $\chi \in \operatorname{Irr}(W)$, we then define the \defn{unipotent uniform almost character} 
\[
    R_{\chi} := \frac{1}{|W|}\sum_{w \in W} \chi(w)R_w,
\]
which satisfies $R_\chi(1) = \operatorname{Feg}_\chi(q)$.

The elements of the Harish-Chandra series $\mathcal{E}(\mathbf{G}^F,(\mathbf{T}_0^F,1))$ are called the \defn{unipotent principal series characters} of $\mathbf{G}^F$. A choice of a square root of $q$ determines a canonical bijection (see \cite[Example 3.2.6]{geck2020character})
\[
    \operatorname{Irr}(W) \to \mathcal{E}(\mathbf{G}^F,(\mathbf{T}_0^F,1)) \qquad \chi \mapsto \rho_\chi.
\]

\begin{proposition}
\label{prop:principalGeneric}
    Every $\rho_\chi \in \mathcal{E}(\mathbf{G}^F,(\mathbf{T}_0^F,1))$ satisfies $\rho_\chi(1) = \operatorname{Deg}_\chi(q)$.
\end{proposition}
\begin{proof}
    We will include an outline of the proof following \cite[\S 6.1-6.3]{digne2020representations} and \cite[\S 4.2-4.3]{geck2011representations}, but see also \cite[Theorem 3.2.18]{geck2020character}.

    As described in \cite[Example 3.2.6]{geck2020character}, the spetsial Hecke algebra $\mathcal{H}_q(W)$ ``evaluated'' at $q$ is isomorphic to $\operatorname{End}_{\mathbf{G}^F}\left(R_{\mathbf{T}_0}^\mathbf{G}(1)\right)^{\operatorname{opp}}$. The Hom functor
    \[
        \mathbb{C}\mathbf{G}^F\text{-mod} \to \mathcal{H}_q(W)\text{-mod}, \qquad Y \mapsto \operatorname{Hom}_{\mathbf{G^F}}\left(R_{\mathbf{T}_0}^{\mathbf{G}}(1),Y\right)
    \]
    induces (see \cite[Theorem 3.1.18]{geck2020character}) a bijection $\mathcal{E}(\mathbf{G}^F,(\mathbf{T}_0^F,1)) \to \operatorname{Irr}(\mathcal{H}_q(W))$. Composing this with the bijection from Tits' deformation theorem gives the correspondence $\chi \mapsto \rho_\chi$. As a $\mathcal{H}_q(W) \times \mathbb{C}\mathbf{G}^F$-module, $R_{\mathbf{T}_0}^{\mathbf{G}}(1)$ decomposes as 
    \[
        R_{\mathbf{T}_0}^{\mathbf{G}}(1) = \bigoplus_{\chi \in \operatorname{Irr}(W)} \chi_q \otimes \rho_\chi.
    \]
    In particular, the multiplicity of $\chi_q$ is equal to $\rho_\chi(1)$ \cite[Theorem 6.1.1]{digne2020representations}. Also, note that as a character of $\mathbf{G}^F$, we have $R_{\mathbf{T}_0}^{\mathbf{G}}(1) = \sum_{\chi \in \operatorname{Irr}(W)} \chi(1) \rho_\chi$.

    Let $\mathbf{B}_0$ be an $F$-stable Borel containing $\mathbf{T}_0$. Then $R_{\mathbf{T}_0}^{\mathbf{G}}(1) \cong \mathbb{C}[\mathbf{G}^F/\mathbf{B}_0^F]$ is the permutation representation of $\mathbf{G}^F$ on the cosets of $\mathbf{B}_0^F$. Fix a set of representatives $\{\bar{w} \in N_{\mathbf{G}}(\mathbf{T}_0)^F : w \in W\}$. The basis $\{T_w : w \in W\}$ of $\mathcal{H}_q(W)$ can be given explicitly using the Bruhat decomposition by defining $T_w : \mathbb{C}[\mathbf{G}^F / \mathbf{B}_0^F] \to \mathbb{C}[\mathbf{G}^F / \mathbf{B}_0^F]$ by
    \[
        T_w\left(x\mathbf{B}_0^F\right) = \text{ sum of all cosets $y\mathbf{B}_0^F \in \mathbb{C}[\mathbf{G}^F / \mathbf{B}_0^F]$ such that $x^{-1}y \in \mathbf{B}_0^F \bar{w} \mathbf{B}_0^F$}.
    \]
    From this, one can show that $\psi : T_w \mapsto \operatorname{Trace}\left( T_w \,\,\vert\,\, R_{\mathbf{T}_0}^{\mathbf{G}}(1)\right)$ is a multiple of the symmetrizing trace $\tau_q$ with $\psi(T_1) = \operatorname{dim} R_{\mathbf{T}_0}^{\mathbf{G}}(1) = [\mathbf{G}^F : \mathbf{B}^F]$ and $\psi(T_w) = 0$ for $w \neq 1$ (see the proof of \cite[Lemma 4.3.2]{geck2011representations}). Then by \cite[Corollary 6.3.10]{digne2020representations} for each $\chi \in \operatorname{Irr}(W)$ the multiplicity of $\chi_q$ in $R_{\mathbf{T}_0}^{\mathbf{G}}(1)$ is $[\mathbf{G}^F : \mathbf{B}^F] / S_\chi(q)$. Hence $\rho_\chi(1) = [\mathbf{G}^F : \mathbf{B}^F] / S_\chi(q)$. Since $\rho_{\bm{1}}(1) = 1$, it follows that $\operatorname{Deg}_\chi(q) = S_{\bm{1}}(q) / S_\chi(q) = \rho_\chi(1)$.
\end{proof}

\subsection{Families of unipotent characters}

Denote by $\varepsilon$ the sign character of $W$. Define a relation $\preceq$ on $\operatorname{Irr}(W)$ inductively as follows. For $W = \{1\}$, the trivial character is related to itself. Now assume that $W \neq \{1\}$ and that $\preceq$ has been defined for all proper parabolic subgroups of $W$. For $\phi, \phi' \in \operatorname{Irr}(W)$, we write $\phi \preceq \phi'$ if there is a sequence $\phi = \phi_0, \phi_1, \dots, \phi_m = \phi'$ such that for each $i \in \{1,\dots,m\}$ there exists a subset $I_i \subsetneq S$ and $\psi_i, \psi_i' \in \operatorname{Irr}(W_{I_i})$, where $\psi_i \preceq \psi_i'$, such that either
\[
    \langle \operatorname{Ind}_{I_i}^S(\psi_i), \phi_{i-1} \rangle \neq 0, \quad \langle \operatorname{Ind}_{I_i}^S(\psi_i'), \phi_i \rangle \neq 0, \quad \text{and} \quad a_{\psi_i'} = a_{\phi_i},
\]
or
\[
    \langle \operatorname{Ind}_{I_i}^S(\psi_i), \varepsilon \phi_{i} \rangle \neq 0, \quad \langle \operatorname{Ind}_{I_i}^S(\psi_i'), \varepsilon\phi_{i-1} \rangle \neq 0, \quad \text{and} \quad a_{\psi_i'} = a_{\varepsilon \phi_{i-1}}.
\]

\begin{definition}
\label{def:weylFam}
    We say that $\phi, \phi'$ belong to the same \defn{family} in $\operatorname{Irr}(W)$ if $\phi \preceq \phi'$ and $\phi' \preceq \phi$.
\end{definition}

One can show \cite[Proposition 4.1.19]{geck2020character} that the functions $\phi \mapsto a_\phi$ and $\phi \mapsto A_\phi$ are constant on the families of $\operatorname{Irr}(W)$. Moreover, in each family $\mathscr{F} \subseteq \operatorname{Irr}(W)$, there is a unique special character $\phi \in \mathscr{F}$ with $a_\phi = b_\phi$ \cite[Proposition 4.1.20]{geck2020character}.

Define a graph on the set of vertices $\operatorname{Uch}(\mathbf{G}^F)$ as follows: two unipotent characters $\rho_1, \rho_2 \in \operatorname{Uch}(\mathbf{G}^F)$ are joined if and only if there is an irreducible character $\chi \in \operatorname{Irr}(W)$ such that $\langle R_{\chi}, \rho_i \rangle \neq 0$ for $i = 1,2$. The sets of vertices corresponding to the connected components of the graph are called the \defn{families} in $\operatorname{Uch}(\mathbf{G}^F)$. There is another perspective on these families via their unipotent support \cite{geck2000existence}.

There is an induced partition of $\operatorname{Irr}(W)$ wherein two characters $\chi, \chi' \in \operatorname{Irr}(W)$ are equivalent if and only if $R_{\chi}$ and $R_{\chi'}$ have unipotent constituents lying in the same family of $\operatorname{Uch}(\mathbf{G}^F)$. This recovers the families of $\operatorname{Irr}(W)$.

Alternatively \cite[\S 12]{carter1985finite}, the families of $\operatorname{Irr}(W)$ can be recovered from the correspondence $\chi \to \rho_\chi$ identifying the irreducible characters of $\operatorname{Irr}(W)$ with the unipotent principle series characters. That is, $\chi, \chi' \in \operatorname{Irr}(W)$ are in the same family of $\operatorname{Irr}(W)$ if and only if $\rho_\chi,\rho_{\chi'}$ are in the same family of $\operatorname{Uch}(\mathbf{G}^F)$. Moreover, the family of $\operatorname{Uch}(\mathbf{G}^F)$ corresponding to $R_\chi$ is the same as the family corresponding to $\rho_\chi$. That is, $\rho_\chi$ is in the family of $\operatorname{Uch}(\mathbf{G}^F)$ containing the components of $R_\chi$.

Each $\rho \in \operatorname{Uch}(\mathbf{G}^F)$ has a \defn{degree polynomial} $\mathbb{D}_\rho \in \mathbb{R}[x]$, such that $\rho(1) = \mathbb{D}_\rho(q)$. See \cite[definition 2.3.25]{geck2020character} for a precise definition. For the unipotent uniform almost characters, the degree polynomial is the corresponding fake degree. For the unipotent principal series characters, the degree polynomial is the corresponding generic degree. We then define
\begin{align*}
    A_\rho &:= \text{degree of $x$ in } \mathbb{D}_\rho \\
    a_\rho &:= \text{largest non-negative integer such that $x^{a_\rho}$ divides } \mathbb{D}_\rho.
\end{align*}

\begin{proposition}[\cite{geck2020character}, Proposition 4.2.7]
    The $a$ and $A$ values are constant on each family of $\operatorname{Uch}(\mathbf{G}^F)$ and agree with the $a$ and $A$ values of the corresponding family of $\operatorname{Irr}(W)$.
\end{proposition}

\section{Spetses}
\label{sec:spetses}

\subsection{Unipotent characters and Generic Degrees for \texorpdfstring{$G(m,1,n)$}{G(m,1,n)}}

The irreducible representations of $G(m,1,n)$ can be parametrized by $m$-partitions of $n$, that is, tuples $\bm\lambda = (\lambda^{(0)},\lambda^{(1)},\dots,\lambda^{(m-1)})$ of partitions such that $|\lambda^{(0)}| + \cdots + |\lambda^{(m-1)}| = n$. We will generally write the parts of a partition in decreasing order: $\lambda = (\lambda_1 \geq \lambda_2 \geq \cdots)$. 

The group $G(m,1,n)$ is generated by the $n$ elements $\{t,s_1,s_2,\dots,s_{n-1}\}$, where $t$ is given by the matrix $\operatorname{Diag}(\zeta_m,1,\dots,1)$ and $s_i$ is given by the permutation matrix corresponding to the transposition $(i,i+1)$ (see \cite[\S 2.7]{lehrer2009unitary}). The parametrization of irreducible characters of $G(m,1,n)$ is as follows:
\[
    \bm\lambda \leftrightarrow \chi_{\bm\lambda} = \operatorname{Ind}_{G(m,1,|\lambda^{(0)}|) \times \cdots \times G(m,1,|\lambda^{(m-1)}|)}^{G(m,1,n)} \left((\chi_0 \otimes \gamma_0) \boxtimes \cdots \boxtimes (\chi_{m-1} \otimes \gamma_{m-1})\right),
\]

where
\begin{itemize}
    \item $\gamma_k$ is the linear character of $G(m,1,|\lambda^{(k)}|)$ defined by $t \mapsto \zeta_m^k$ and $s_i \mapsto 1$ for $i = 1,\dots,|\lambda^{(k)}| - 1$.
    \item $\chi_k$ is the character of the symmetric group $S_{|\lambda^{(k)}|}$ corresponding to $\lambda^{(k)}$ (see \cite{sagan2013symmetric} for instance) considered as a character of $G(m,1,|\lambda^{(k)}|)$ via the surjection $G(m,1,|\lambda^{(k)}|) \to S_{|\lambda^{(k)}|}$.
\end{itemize}

Every irreducible character of $G(m,1,n)$ is equal to $\chi_{\bm\lambda}$ for some $\bm\lambda$, and $\chi_{\bm\alpha} = \chi_{\bm\beta}$ if and only if $\bm\alpha = \bm\beta$. We will describe explicit models for the irreducible representations in terms of the $m$-partitions of $n$ as given in \cite{ariki1994hecke} (although the precise notation will follow \cite{marin2010automorphisms}).

Fix $\bm\lambda$, an $m$-partition of $n$. Denote by $\mathcal{T}(\bm\lambda)$ the set of standard Young tableaux with shape $\bm\lambda$. That is, the elements of $\mathcal{T}(\bm\lambda)$ are tuples $\mathbf{T} = (T_0,\dots,t_{m-1})$ such that $T_i$ is a filling of the Young diagram of $\lambda^i$ with numbers in $\{1,\dots,n\}$ such that the numbers increase across the rows and columns of each $T_i$, and each number appears exactly once in one of the filled diagrams. We will let $\mathbf{T}(k) = i$ if $k$ is placed in $T_i$. Let $V(\bm\lambda)$ be the $\mathbb{C}$-vector space with basis $\mathcal{T}(\bm\lambda)$, and define the representation $\rho_{\bm\lambda} : G(m,1,n) \to V(\bm\lambda)$ by
\begin{itemize}
    \item $\rho_{\bm\lambda}(t)\mathbf{T} = \zeta_m^{\mathbf{T}(1)}\mathbf{T}$.
    \item Let $\mathbf{T}_{i \leftrightarrow i+1}$ be the tuple obtained by exchanging the numbers $i$ and $i+1$ if this results in a tuple of standard Young tableaux, and 0 otherwise. If $i$ and $i+1$ occur in the same tableau, say $i$ is in row $\ell_1$ column $m_1$ and $i+1$ is in row $\ell_2$ column $m_2$, define the axial distance $a(i,i+1)$ between them to be $a(i,i+1) = (m_2 - \ell_2) - (m_1 - \ell_1)$. If $i$ and $i+1$ do not occur in the same tableau, then set $a(i,i+1) = \infty$ so that $1/a(i,i+1) = 0$. Then
    \[
        \rho_{\bm\lambda}(s_i)\mathbf{T} = \frac{1}{a(i,i+1)}\mathbf{T} + \left( 1 + \frac{1}{a(i,i+1)} \right)\mathbf{T}_{i\leftrightarrow i+1}.
    \]
\end{itemize}
It is easy to see from this construction that the reflection representation is given by the $m$-partition $(n-1,1,\varnothing, \dots, \varnothing)$.

\begin{example}[Exterior powers]
    Let $V$ be the reflection representation of $G(m,1,n)$. We will show that the $m$-partition of $n$ corresponding to $\Lambda^k V$ is $(n-k,1^k,\varnothing,\dots,\varnothing)$.

    Let $\bm{\lambda}(k)$ be the partition $(n-k,1^k,\varnothing,\dots,\varnothing)$. Consider the model $V(\bm{\lambda}(1))$ given above for the reflection representation with basis $\mathcal{T}(\bm{\lambda}(1))$ consisting of standard Young tableaux $\mathbf{T} = (T_0,T_1,\dots,T_m)$ of shape $\bm{\lambda}(1)$. Let $\mathbf{T}^{a_1,a_2,\dots,a_k}$ be the standard Young tableaux of shape $\bm{\lambda}(k)$, with
    \[
        T^{a_1,a_2,\dots,a_k}_1 = \begin{ytableau}
            a_1 \\
            a_2 \\
            \vdots \\
            a_k
        \end{ytableau}
    \]
    for $1 \leq a_1 < a_2 < \cdots < a_k \leq n$. So $\mathcal{T}(\bm{\lambda}(k)) = \{\mathbf{T}^{a_1,\dots,a_k} : 1 \leq a_1 < \cdots < a_k \leq n\}$. Define a linear map $\varphi : \Lambda^k V(\bm{\lambda}(1)) \to V(\bm{\lambda}(k))$ by
    \[
        \mathbf{T}^{a_1} \wedge \mathbf{T}^{a_2} \wedge \cdots \wedge \mathbf{T}^{a_k} \mapsto \mathbf{T}^{a_1,\dots,a_k}, 
    \]
    for $1 \leq a_1 < \cdots < a_k \leq n$. We need to show that $\varphi$ gives an isomorphism of representations. Now
    \begin{align*}
        \varphi(t \cdot (\mathbf{T}^{a_1} \wedge \cdots \wedge \mathbf{T}^{a_k})) &= \varphi\left(\zeta_m^{\mathbf{T}^{a_1}(1) + \cdots + \mathbf{T}^{a_k}(1)}(\mathbf{T}^{a_1} \wedge \cdots \wedge \mathbf{T}^{a_k})\right) \\
        &= \zeta_m^{\mathbf{T}^{a_1,\dots,a_k}(1)}\mathbf{T}^{a_1,\dots,a_k} \\
        &= t \cdot \mathbf{T}^{a_1,\dots,a_k}.
    \end{align*}
    For each $i = 1,\dots,n-1$, there are three cases to consider:
    \begin{enumerate}
        \item $\mathbf{T}^{a_1,\dots,a_k}(i) = \mathbf{T}^{a_1,\dots,a_k}(i+1) = 0$. In this case, 
        \begin{align*}
            \varphi(s_i \cdot (\mathbf{T}^{a_1} \wedge \cdots \wedge \mathbf{T}^{a_k})) &= \varphi(\mathbf{T}^{a_1} \wedge \cdots \wedge \mathbf{T}^{a_k}) \\
            &= \mathbf{T}^{a_1,\dots,a_k} = s_i \cdot \mathbf{T}^{a_1,\dots,a_k}.
        \end{align*}

        \item $\mathbf{T}^{a_1,\dots,a_k}(i) = \mathbf{T}^{a_1,\dots,a_k}(i+1) = 1$. In this case,
        \begin{align*}
            \varphi(s_i \cdot (\mathbf{T}^{a_1} \wedge \cdots \mathbf{T}^{i} \wedge \mathbf{T}^{i+1} \wedge \cdots \wedge \mathbf{T}^{a_k})) &= \varphi(\mathbf{T}^{a_1} \wedge \cdots \mathbf{T}^{i+1} \wedge \mathbf{T}^{i} \wedge \cdots \wedge \mathbf{T}^{a_k}) \\
            &= - \mathbf{T}^{a_1,\dots,a_k} = s_i \cdot \mathbf{T}^{a_1,\dots,a_k}.
        \end{align*}

        \item $\mathbf{T}^{a_1,\dots,a_k}(i) \neq \mathbf{T}^{a_1,\dots,a_k}(i+1)$. That is, exactly one of $i$, $i+1$ is in $\{a_1,\dots,a_k\}$, say $a_j = i$ or $i+1$. In this case,
        \begin{align*}
            \varphi(s_i \cdot (\mathbf{T}^{a_1} \wedge \cdots \wedge \mathbf{T}^{a_j} \wedge \cdots \wedge \mathbf{T}^{a_k})) &= \varphi (\mathbf{T}^{a_1} \wedge \cdots \wedge \mathbf{T}^{a_j}_{i \leftrightarrow i+1} \wedge \cdots \wedge \mathbf{T}^{a_k}) \\
            &= \mathbf{T}^{a_1,\dots,a_k}_{i \leftrightarrow i+1} = s_i \cdot \mathbf{T}^{a_1,\dots,a_k}.
        \end{align*}
        
        Thus, the $m$-partition of of $n$ corresponding to $\Lambda^k V$ is $(n-k,1^k,\varnothing,\dots,\varnothing)$.
    \end{enumerate}
\end{example}

\begin{example}[Galois twists]
\label{ex:galoisPart}
    One can easily show, using Schur orthogonality relations, that the elements of $\operatorname{Gal}(\mathbb{Q}(\zeta_h) / \mathbb{Q})$ act as permutations on $\operatorname{Irr}(W)$ for $W$ a well-generated irreducible complex reflection group with Coxeter number $h$. For the groups $G(m,1,n)$, we will characterize this action in terms of the $m$-partitions of $n$.

    The following is \cite[Lemma 5.2]{marin2010automorphisms} (which contains a small typo). For $\bm{\lambda} = (\lambda^{(0)},\lambda^{(1)},\dots,\lambda^{(m-1)})$, let $\sigma_p(\bm{\lambda}) = (\lambda^{(0)} , \lambda^{(p^{-1})}, \dots, \lambda^{((m-1)p^{-1})})$, where the indices are taken mod $m$ and $p^{-1}$ is the multiplicative inverse of $p$ mod $mn$. That is, $[\sigma_p(\bm{\lambda})]^{(pk)} = \lambda^{(k)}$ (indices taken mod $m$). Then $V(\bm{\lambda})^{\sigma_p} \cong V(\sigma_p(\bm{\lambda}))$.

    To see this, define $\varphi : V(\bm{\lambda})^{\sigma_p} \to V(\sigma_p(\bm{\lambda}))$ by
    \[
        \mathbf{T} \mapsto \mathbf{T}^{\sigma_p}, \quad \text{where} \quad T^{\sigma_p}_{pk} = T_k.
    \]
    This is an isomorphism of representations because
    \[
        \varphi(t \cdot \mathbf{T}) = \varphi(\zeta_m^{p\mathbf{T}(1)} \mathbf{T}) = \zeta_m^{p\mathbf{T}(1)} \mathbf{T}^{\sigma_p} = t \cdot \mathbf{T}^{\sigma_p}
    \]
    and
    \[
        \varphi(s_i \cdot \mathbf{T}) = \frac{1}{a(i,i+1)}\mathbf{T}^{\sigma_p} + \left(1 + \frac{1}{a(i,i+1)}\right)\mathbf{T}^{\sigma_p}_{i \leftrightarrow i+1} = s_i \cdot \mathbf{T}^{\sigma_p}.
    \]

    One can see that $\Lambda^k V^{\sigma_p} \cong (\Lambda^k V)^{\sigma_p}$ as representations of $W$ by noticing that they induce the same character. So the $m$-partition of $n$ corresponding to $\Lambda^k V^{\sigma_p}$ is $(n-k,\varnothing,\dots,\varnothing,1^k,\varnothing,\dots,\varnothing)$, where the $1^k$ is in the $p$th slot (mod $m$; indexing starts with 0).
\end{example}

In \cite{malle1995unipotente}, Malle gives a combinatorial construction of unipotent characters and generic degrees for the groups $G(m,1,n)$ and $G(m,m,n)$ which enjoy many of the same properties as the corresponding objects for Weyl groups. We will now describe the construction for the groups $G(m,1,n)$. The construction for the groups $G(m,m,n)$ will be given later.

\begin{definition}
    An \defn{$m$-symbol} is a tuple $S = (S_0,\dots,S_{m-1})$ of $m$ finite sequences $S_i = (0 \leq S_{i,1} < S_{i,2} < \cdots < S_{i,M_i})$ of strictly increasing non-negative integers. An $m$-symbol will be presented as
    \[
        S = \begin{pmatrix}
            S_{0,1} & \cdots & S_{0,M_0} \\
            \vdots & & \vdots \\
            S_{m-1,1} & \cdots & S_{m-1,M_{m-1}}
        \end{pmatrix}.
    \]
\end{definition}

The \defn{content} of an $m$-symbol $S$ is $I(S) = M_0 + \cdots + M_{m-1}$. Its \defn{rank} is
\[
    \operatorname{rk}(S) = \sum_{i,j} \lambda_{i,j} - \left\lfloor \frac{(I(S) - 1)(I(S) - m + 1)}{2m} \right\rfloor.
\]
For an $m$-symbol $S$ with $I(S) \equiv 1 \mod m$, define its \defn{defect} by
\[
    \operatorname{def}(S) = \left(\frac{(m-1)(I(S) - 1)}{2} - \sum_{i=0}^{m-1} i M_i\right) \mod m.
\]
Such a symbol will be called \defn{reduced} if $\operatorname{def}(S) = 0$ and not all of the $\lambda_{i,1}$ are zero.

\begin{definition}
\label{def:uniDeg1}
    The \defn{unipotent characters} of $G(m,1,n)$ are
    \[
        \operatorname{Uch}(G(m,1,n)) := \{ \text{reduced $m$-symbols } S : \operatorname{rk}(S) = m, \,\, I(S) \equiv 1 \mod m\}.
    \]
    We'll say that two unipotent characters $S$, $S'$ belong to the same family if the entries of $S$ and $S'$ coincide as multisets.
    
    The \defn{generic degree} of a unipotent character is  defined by
    \[
        \operatorname{Deg}_S(q) := \frac{(-1)^{\binom{m}{2}\binom{\ell}{2}} \left(\prod_{i=1}^n (q^{mi}-1)\right) \cdot \prod_{i=0}^{m-1} \prod_{j=i}^{m-1} \prod_{\substack{(\lambda, \mu) \in S_i \times S_j \\ \mu < \lambda \text{ if } i = j}} (q^{\lambda} \zeta_m^i - q^{\mu}\zeta_m^j) }{\tau(m)^\ell q^{\binom{m(\ell-1) + 1}{2} + \binom{m(\ell-2) + 1}{2} + \cdots} \prod_{i=0}^{m-1} \Theta (S_i, q^m)}
    \]
    where 
    \[
        \ell = \frac{I(S) - 1}{m}, \qquad \tau(m) = \prod_{i=0}^{m-1} \prod_{j=i+1}^{m-1}(\zeta^i - \zeta^j) = \sqrt{-1}^{\binom{m-1}{2}}\sqrt{m}^m,
    \]
    and for any finite $A \subset \mathbb{Z}$,
    \[
        \Theta(A,q) = \prod_{\substack{\lambda \in A \\ \lambda \geq 1}} \prod_{j=1}^\lambda (q^j - 1).
    \]
\end{definition}

There is an inclusion $\operatorname{Irr}(G(m,1,n)) \subset \operatorname{Uch}(G(m,1,n))$. Let $\bm\lambda = (\lambda^{(0)}, \dots, \lambda^{(m-1)})$ be an $m$-partition of $n$. Let $\alpha_0 = (0 \leq \alpha_{0,1} \leq \cdots \leq \alpha_{0,M+1})$ be the parts of $\lambda^{(0)}$ written in increasing order with 0's added if necessary. Let $\alpha_i = (0 \leq \alpha_{i,1} \leq \cdots \leq \alpha_{i,M})$ be the parts of $\lambda^{(i)}$ written in increasing order for $i > 0$ with 0's added if necessary. Choose $M$ as small as possible. Define an $m$-symbol $S_{\bm\lambda} = (S_0, \dots, S_{m-1})$ by $S_{i,j} = \alpha_{i,j} + j - 1$. Then $S_{\bm\lambda}$ is a reduced symbol with $I(S_{\bm\lambda}) = mM + 1$. The inclusion is defined by
\[
    \operatorname{Irr}(G(m,1,n)) \to \operatorname{Uch}(G(m,1,n)) \qquad \chi_{\bm\lambda} \mapsto S_{\bm\lambda}.
\]
Under this inclusion, the definition of the generic degree of an irreducible character given in \Cref{def:genDeg} coincides with the definition for unipotent characters. That is, $\operatorname{Deg}_{\chi_{\bm\lambda}}(q) = \operatorname{Deg}_{S_{\bm\lambda}}(q)$.

Malle also gives a formula for the fake degree of an irreducible character in terms of the corresponding $m$-symbol. This formula has to be slightly adjusted (following \cite[Proposition 5.3.11]{lasy2012traces}) because Malle uses a different definition of fake degrees (see \Cref{rem:fegDef}).
\begin{proposition}[\cite{malle1995unipotente}, Proposition 2.10]
\label{prop:fegForm1}
    The fake degree of $\chi_{\bm\lambda}$, with corresponding $m$-symbol $S_{\bm\lambda} = (S_0,\dots,S_{m-1})$, is equal to
    \[
        \left(\prod_{i=1}^n (q^{mi}-1)\right)\cdot \left(\prod_{i=0}^{m-1}\frac{\Delta(S_i,q^m)}{\Theta(S_i,q^m)} \right) \cdot \frac{\prod_{i=1}^{m-1}\prod_{\lambda \in S_i}q^{(m-i)\lambda}} {q^{\binom{m(\ell-1)+1}{2} + \binom{m(\ell - 2)+1}{2} + \cdots}},
    \]
    where for any finite $A \subset \mathbb{Z}$,
    \[
        \Delta(A,q) = \prod_{\substack{\lambda,\lambda' \in A \\ \lambda' < \lambda}} (q^\lambda - q^{\lambda'}),
    \]
    and $\ell$ and $\Theta$ are as in \Cref{def:uniDeg1}.
\end{proposition}

\begin{example}
\label{ex:galFeg1}
    The $m$-symbol corresponding to $V^{\sigma_p}$ is given by
    \[
        S_i = \begin{cases}
            (0,n) & \text{if } i = 0, \\
            (1) & \text{if } i = p \mod m, \\
            (0) & \text{otherwise}.
        \end{cases}
    \]
    So
    \[
        \Delta(S_i, q^m) = \begin{cases}
            q^{mn} - 1 & \text{if } i = 0, \\
            1 & \text{otherwise},
        \end{cases} \qquad \text{and} \qquad \Theta(S_i,q^m) = \begin{cases}
            \prod_{j=1}^n(q^{mj} - 1) & \text{if } i = 0, \\
            (q^m-1) & \text{if } i = p \mod m, \\
            1 & \text{otherwise}.
        \end{cases}
    \]
    Letting $\chi_{1,p}$ denote the character of $V^{\sigma_p}$, we then have
    \begin{align*}
        \operatorname{Feg}_{\chi_{1,p}}(q) &= \left(\prod_{i=1}^n (q^{mi}-1)\right) \cdot \frac{q^{m-(p \mod m)}}{(q^m-1)\prod_{j=1}^{n-1}(q^{mj} - 1)} \\
        &= q^{m-(p \mod m)} + q^{2m-(p \mod m)} +\cdots + q^{mn-(p \mod m)}.
    \end{align*}
    Since $p$ is relatively prime to $mn$, we can rewrite this as
    \[
        \operatorname{Feg}_{\chi_{1,p}}(q) = \sum_{j=1}^n q^{(mj - p) \mod mn} = \sum_{j=1}^n q^{p(mj-1) \mod mn} = \sum_{j=1}^n q^{pe_j \mod h},
    \]
    which demonstrates \Cref{ex:galoisFeg}.
\end{example}

The Schur elements of irreducible representations of $G(m,1,n)$ can be expressed in terms of the corresponding $m$-partition using the next result.

\begin{theorem}[\cite{mathas2004matrix}, Corollary 5.16]
\label{thm:imprimSchur}
    For partitions $\lambda, \mu$ (whose parts are written in decreasing order), define
    \begin{align*}
        [\lambda] &:= \{(i,j) : i \geq 1, \,\, 1 \leq j \leq \lambda_i\}, \\
        \lambda'_k &:= \#\{i : i \geq 1 \text{ such that } \lambda_i \geq k\}, \\
        n(\lambda) &:= \sum_{i \geq 1} (i-1)\lambda_i, \\
        h_{i,j}^{\lambda, \mu} &:= \lambda_i - i + \mu'_j - j + 1 \quad \text{for } (i,j) \in [\lambda].
    \end{align*}
    Let $\bm\lambda = (\lambda^{(0)},\lambda^{(1)},\dots,\lambda^{(m-1)})$ be the $m$-partition of $n$ corresponding to the irreducible representation $\chi$. Then
    \[
        S_\chi(q) = (-1)^{mn} q^{-\alpha(\bm\lambda)-n} \prod_{s=0}^{m-1} \prod_{(i,j) \in [\lambda^{(s)}]} Q_s \left[h_{i,j}^{\lambda^{(s)}, \lambda^{(s)}} \right]_q \prod_{t=s+1}^{m-1} X_{s,t}^{\bm\lambda},
    \]
    where
    \[
        X_{s,t}^{\bm\lambda} = \left(\prod_{(i,j) \in [\lambda^{(t)}]}(q^{j-i}Q_t - Q_s)\right) \cdot \left(\prod_{(i,j) \in [\lambda^{(s)}]} \left(q^{j-i}Q_s - q^{\lambda_1^{(t)}} Q_t\right) \prod_{k=1}^{\lambda_1^{(t)}} \frac{q^{j-i}Q_s - q^{k-1-\lambda_k^{(t)'}}Q_t}{q^{j-i} Q_s - q^{k-\lambda_k^{(t)'}}Q_t} \right),
    \]
    and
    \[
        Q_s = \begin{cases}
            q & \text{if } s = 0 \\
            \zeta_m^s & \text{otherwise}.
        \end{cases} \qquad \text{ and } \qquad \alpha(\bm\lambda) = \sum_{0 \leq s \leq m-1} n(\lambda^{(s)}).
    \]
\end{theorem}

We will want to be able to evaluate the Schur elements at the values $\zeta_{mn}^p$ for $p$ relatively prime to $mn$. For this, Mathas's formula will be difficult to work with, so we will prefer the following simplification:

\begin{theorem}[\cite{chlouveraki2012schur}, Theorem 3.2]
\label{thm:imprimSchurSimp}
    Let $\bm\lambda = (\lambda^{(0)},\lambda^{(1)},\dots,\lambda^{(m-1)})$ be the $m$-partition of $n$ corresponding to the irreducible representation $\chi$, and let $\bar{\bm\lambda}$ be the (decreasing) sequence consisting of all of the parts of $\lambda^{(0)},\lambda^{(1)},\dots,\lambda^{(m-1)}$ combined. Then
    \[
        S_\chi(q) = (-1)^{n(m-1)}q^{-n(\bar{\bm\lambda})}(q-1)^{-n} \prod_{s=0}^{m-1} \prod_{(i,j) \in [\lambda^{(s)}]} \prod_{t = 0}^{m-1} \left(q^{h_{i,j}^{\lambda^{(s)}, \lambda^{(t)}}} Q_sQ_t^{-1} - 1 \right),
    \]
    where
    \[
        Q_s = \begin{cases}
            q & \text{if } s = 0 \\
            \zeta_m^s & \text{otherwise}.
        \end{cases}
    \]
\end{theorem}

We are now ready to prove \Cref{thm:galDeg} for the spetsial groups $G(m,1,n)$.

\begin{proof}[Proof of \Cref{thm:galDeg}]
    We will first compute the generic degrees for the exterior powers of Galois twists of the reflection representation. We saw in \Cref{ex:galoisPart} that the $m$-partition of $n$ corresponding to $\Lambda^k V^{\sigma_p}$ is $\bm\lambda = (n-k,\varnothing,\dots,\varnothing,1^k,\varnothing,\dots,\varnothing)$. Now 
    \[
        h_{1,j}^{\lambda^{(0)}, \lambda^{(t)}} = \begin{cases}
            n-k-j+1 & \text{if } t = 0 \\
            n-1 & \text{if } t=p \mod m \text{ and } j = 1 \\
            n-k -j & \text{otherwise},
        \end{cases} \quad \text{ and } \quad h_{i,1}^{\lambda^{(p)}, \lambda^{(t)}} = \begin{cases}
            2 - i & \text{if } t=0 \\
            k-i+1 & \text{if } t = p \mod m \\
            1 - i & \text{otherwise}.
        \end{cases}
    \]
    So, letting $\chi$ be the character corresponding to $\Lambda^k V^{\sigma_p}$,
    \begin{align*}
        S_{\chi}(q) &= (-1)^{n(m-1)}q^{-k(k+1)/2}(q-1)^{-n} \left(\prod_{j=1}^{n-k}\left(q^{n-k-j+1} - 1 \right)\right)  \left(\frac{q^{n}\zeta_m^{-p} - 1}{q^{n-k}\zeta_m^{-p}-1} \right) \\
        & \qquad \left(\prod_{t=1}^{m-1} \prod_{j=1}^{n-k} \left(q^{n-k-j+1}\zeta_m^{-t} -1\right)\right) \left(\prod_{i=1}^k \left(q^{k-i+1} - 1\right)\right) \left(\prod_{\substack{t=0 \\ t \neq p \mod m}}^{m-1}\prod_{i=1}^k \left(q^{1-i}\zeta_m^{p-t} - 1\right) \right) \\
        &= \frac{(-1)^{n(m-1)}(q^n - \zeta_m^p)}{q^{k(k+1)/2}(q-1)^{n}(q^{n-k} - \zeta_m^p)}\left(\prod_{j=1}^{n-k}\left(q^{j} - 1 \right)\right) \left(\prod_{j=1}^{n-k} (-1)^{m-1} \frac{q^{mj} - 1}{q^j - 1} \right) \\
        & \qquad \left(\prod_{i=1}^k \left(q^{i} - 1\right)\right) \left((-1)^{m-1}\prod_{\substack{t=0 \\ t \neq p \mod m}}^{m-1} \left(1-\zeta_m^{p-t}\right) \right) \left(\prod_{i=1}^{k-1} (-1)^{m-1} q^{-i(m-1)} \frac{q^{mi}-1}{q^i -1} \right) \\ 
        &= \frac{m(q^k-1)(q^n - \zeta_m^p)\prod_{j=1}^{n-k}(q^{mj}-1)\prod_{i=1}^{k-1}(q^{mi} - 1)}{q^{k + m\binom{k}{2}}(q-1)^n(q^{n-k}-\zeta_m^p)}.
    \end{align*}
    The Poincare polynomial for $G(m,1,n)$ is $\prod_{i=1}^n \frac{q^{mi}-1}{q-1}$, so
    \begin{align}
        \operatorname{Deg}_{\chi}(q) &= \frac{q^{k + m\binom{k}{2}}(q^{n-k}-\zeta_m^p)\prod_{i=k}^n(q^{mi}-1)}{m(q^k-1)(q^n - \zeta_m^p)\prod_{j=1}^{n-k}(q^{mj}-1)} \label{eqn:deg} \\
        &= \frac{q^{k + m\binom{k}{2}}(q^{n-k}-\zeta_m^p)\prod_{i=k}^{n-1}(q^{mi}-1) }{m(q^k-1)\prod_{j=1}^{n-k}(q^{mj}-1)}\prod_{\substack{j=0 \\ j \neq p \mod m}}^{m-1}(q^n - \zeta_m^j). \nonumber
    \end{align}
    Hence
    \begin{align*}
        \operatorname{Deg}_{\chi}(\zeta_{mn}^p) &= \frac{\zeta_{mn}^{pk}\zeta_n^{p\binom{k}{2}}(\zeta_{mn}^{-pk}\zeta_m^p-\zeta_m^p)\prod_{i=k}^{n-1}(\zeta_{n}^{pi}-1) }{m(\zeta_{mn}^{pk}-1)\prod_{j=1}^{n-k}(\zeta_{n}^{pj}-1)}\prod_{\substack{j=0 \\ j \neq p \mod m}}^{m-1}(\zeta_{m}^{p} - \zeta_m^j) \\
        &= \frac{\zeta_n^{p\binom{k}{2}}\zeta_m^p(1-\zeta_{mn}^{pk})\prod_{i=k}^{n-1}(\zeta_{n}^{pi}-1) }{m(\zeta_{mn}^{pk}-1)\prod_{j=1}^{n-k}(-\zeta_n^{pj})(\zeta_{n}^{p(n-j)}-1)} (\zeta_m^{-p})\prod_{j=1}^{m-1}(1 - \zeta_m^j) \\
        &= \frac{-\zeta_n^{p\binom{k}{2}}}{(-1)^{n-k}\zeta_n^{p\binom{n-k+1}{2}}}\frac{\prod_{i=k}^{n-1}(\zeta_{n}^{pi}-1)}{\prod_{j=1}^{n-k}(\zeta_{n}^{p(n-j)}-1)}\frac{\prod_{j=1}^{m-1}(1 - \zeta_m^j)}{m} \\
        &= (-1)^{n-k+1}\zeta_n^{p\left(kn - \frac{n(n+1)}{2} \right)} = (-1)^k.
    \end{align*}
    
    Now let $\bm\lambda$ be any $m$-partition of $n$. Notice that the Poincare polynomial evaluated at $\zeta_{mn}^p$ is always 0. So if the generic degree of the representation corresponding to $\lambda$ evaluated at $\zeta_{mn}^p$ is nonzero, then the Schur element evaluated at $\zeta_{mn}^p$ must be 0. Looking at the formula in \Cref{thm:imprimSchurSimp}, notice that the Schur element vanishes at $\zeta_{mn}^p$ if and only if one of the factors $\left(q^{h_{i,j}^{\lambda^{(s)}, \lambda^{(t)}}} Q_sQ_t^{-1} - 1 \right)$ vanishes at $\zeta_{mn}^p$. This happens if and only if at least one of the following is a multiple of $mn$:
    \begin{enumerate}
        \item $ph_{i,j}^{\lambda^{(0)},\lambda^{(0)}}$ 
        \item $p(h_{i,j}^{\lambda^{(0)}, \lambda^{(t)}} + 1) - nt$
        \item $p(h_{i,j}^{\lambda^{(s)}, \lambda^{(0)}} - 1) + ns$ 
        \item $ph_{i,j}^{\lambda^{(s)}, \lambda^{(t)}} + n(s-t)$,
    \end{enumerate}
    where $s,t \geq 1$. 
    
    Expression (1) is a multiple of $mn$ if and only if $h_{i,j}^{\lambda^{(0)}, \lambda^{(0)}}$ is a multiple of $mn$, but this is not possible because $1 \leq h_{i,j}^{\lambda^{(0)}, \lambda^{(0)}} \leq n$. 

    If expression (2) is a multiple of $mn$, then $(h_{i,j}^{\lambda^{(0)}, \lambda^{(t)}} + 1)$ is a multiple of $n$. Now $-(n-1) \leq h_{i,j}^{\lambda^{(0)}, \lambda^{(t)}} \leq n-1$, so either $h_{i,j}^{\lambda^{(0)}, \lambda^{(t)}} = -1$ or $n-1$. If it equals $-1$, then $-nt$ is a multiple of $mn$, which is not possible because $1 \leq t \leq m-1$. So $h_{i,j}^{\lambda^{(0)}, \lambda^{(t)}} = n-1$, which can only happen if $i=j=1$ and $\lambda^{(0)}_1 + \lambda_1^{(t)'} = n$. So for some $1 \leq k \leq n-1$ we have that $\lambda^{(0)} = n-k$ and $\lambda^{(t)} = 1^k$. Moreover, we must have $t = p \mod m$ because $p-t$ needs to be a multiple of $m$. It follows that $\bm\lambda$ corresponds to the $k$th exterior power of the $p$th Galois twist. 

    If expression (3) is a multiple of $mn$, then $(h_{i,j}^{\lambda^{(s)}, \lambda^{(0)}} - 1)$ is a multiple of $n$, so $h_{i,j}^{\lambda^{(s)}, \lambda^{(0)}} = 1$ or $-(n-1)$. If it equals $1$, then $s$ must be a multiple of $m$ which is not possible. So $h_{i,j}^{\lambda^{(s)}, \lambda^{(0)}} = -(n-1)$, which can only happen if $i=n$, $j=1$, $\lambda^{(s)}_i = 1$, and $\lambda_j^{(0)'} = 0$. So we have $\lambda^{(s)} = 1^n$. Moreover, $s = p \mod m$, because $-p + s$ needs to be a multiple of $m$. It follows that $\bm\lambda$ corresponds to the representation $\Lambda^n V^{\sigma_p}$.

    If expression (4) is a multiple of $mn$, then $h_{i,j}^{\lambda^{(s)}, \lambda^{(t)}}$ is a multiple of $n$, which can only happen if $h_{i,j}^{\lambda^{(s)}, \lambda^{(t)}} = 0$. So $s-t$ needs to be a multiple of $m$, which implies that $s = t$. But if $s = t$, then $h_{i,j}^{\lambda^{(s)}, \lambda^{(t)}} \geq 1$, so we have a contradiction.

    Thus we can conclude that $\operatorname{Deg}_\chi(\zeta_{mn}^p) = 0$ if $\chi$ is not an exterior power of $V^{\sigma_p}$.
\end{proof}

\subsection{Clifford Theory}

\begin{theorem}[Clifford's Theorem, \cite{clifford1937representations}]
    Let $k$ be a field, $G$ a group, and $N$ a normal subgroup of finite index in $G$. Suppose that $V$ is a simple $kG$-module. Fix a simple summand $W$ of $\operatorname{Res}^G_N(V)$, and define the inertia subgroup, $T(W)$, to be the set of all $g \in G$ such that the $kN$-module $W^g = \{g \cdot w : w \in W\}$ is isomorphic to $W$. If $X$ is a transversal to $T(W)$ in $G$, then there is a positive integer $e$ (called the ramification index of $V$ in $N$) such that
    \[
        \operatorname{Res}^G_N(V) \cong \bigoplus_{x \in X} (W^x)^{\oplus e}.
    \]
\end{theorem}

Clifford's theorem has led to an area of representation called Clifford theory that relates representations of $N$ with representations of $G$. See \cite{isaacs2006character} or \cite[\S 7]{craven2019representation} for more on the general theory.

We will now describe the irreducible characters of the groups $G(m,m,n)$ following \cite[\S 3.4.2]{lasy2012traces} to emphasize the role of Clifford theory.

The character version of Clifford's theorem is as follows:
\begin{theorem}[Clifford's Theorem]
    (We will assume that $G$ is finite) Let $H$ be a normal subgroup of $G$, $\phi$ an irreducible character of $G$, and $\chi$ an irreducible character of $H$ such that $\langle \phi_H, \chi \rangle_H \neq 0$. Then 
    \[
        \phi_H = \langle \phi_H, \chi \rangle_H \sum_{g \in G / I_G(\chi)} \chi^g,
    \]
    where $\chi^g(h) = \chi(g^{-1} h g)$ and $I_G(\chi) = \{g \in G : \chi^g(h) = \chi(h) \text{ for all } h \in H\}$.
\end{theorem}

We will also need:
\begin{theorem}[\cite{isaacs2006character}, Theorem 6.11]
\label{thm:char}
    Let $H \trianglelefteq G$, $\chi \in \operatorname{Irr} H$, and $T = I_G(\chi)$. Define
    \[
        A = \{\psi \in \operatorname{Irr} T : \langle \psi_H , \chi \rangle_H \neq 0\}, \qquad B = \{\phi \in \operatorname{Irr} G : \langle \phi_H, \chi \rangle_H \neq 0\}.
    \]
    \begin{enumerate}
        \item If $\psi \in A$ then $\operatorname{Ind}_{T}^G \psi$ is irreducible.
        \item The map $\psi \mapsto \operatorname{Ind}_{T}^G \psi$ is a bijection of $A$ onto $B$.
        \item If $\operatorname{Ind}_T^G \psi = \phi$ with $\psi \in A$, then $\psi$ is the unique irreducible constituent of $\phi_{T}$ which lies in $A$, and $\langle \psi_H, \chi \rangle_H = \langle \phi_H, \chi \rangle_H$.
    \end{enumerate}
\end{theorem}

Suppose that $G = H \rtimes \langle t \rangle$, with $t^m = 1$. Now for $\chi \in \operatorname{Irr} H$, the inertia group is $T = I_{G}(\chi) = H \rtimes \langle t^d \rangle$ for $d$ a divisor of $m$. To see this, note that if $t^k \in T$, then
\[
    \chi((nt^k)^{-1}h(nt^{k})) = \chi(t^{-k}n^{-1}hnt^k) = \chi(n^{-1}hn) = \chi(h) \implies nt^k \in T,
\]
for all $n,h \in H$.
If $V$ is a $\mathbb{C}$-vector space affording the irreducible representation corresponding to $\chi$, then there exists an isomorphism $\varphi_0 : V \to V$ such that $t^{-d}ht^d\varphi_0(v) = \varphi_0(hv)$ for any $h \in H$ and $v \in V$ (since $t^d \in T$). By Schur's lemma, there exists $\lambda \in \mathbb{C}$ such that $\varphi := \lambda \varphi_0$ satisfies $\varphi^{m/d} = \text{Id}_V$ (the exponent here signifies repeated composition). Now define an action of $t^d$ on $V$ by $t^d \cdot v = f^{-1}(v)$. This gives an action of $T$ on $V$ because 
\[
    h_1 f^{-k_1}(h_2 f^{-k_2}(v)) = h_1(t^{dk_1}h_2t^{-dk_1}) f^{-(k_1 + k_2)}(v).
\]
Hence there exists $\psi \in \operatorname{Irr} T$ such that $\psi_H = \chi$.

Let $\bar{\chi}$ be any irreducible character of $T$ such that $\bar{\chi}_H = \chi$ (we have just shown that there is at least one such character). Now, $\operatorname{Ind}_H^T \chi = \sum_{j=1}^{m/d} \bar{\chi} \otimes \xi^j$, where $\xi \in \operatorname{Irr} T$ is the linear character with $\xi(H) = 1$ and $\xi(t^d) = \exp(2d\pi i/m)$ (this can be easily checked using Frobenius reciprocity and comparing degrees). So the $\sum_{j=1}^{m/d} \bar{\chi} \otimes \xi^j$ are precisely the irreducible characters of $T$ which restrict to $\chi$ on $H$, and $\operatorname{Ind}_H^G \chi = \sum_{j=1}^{m/d} \operatorname{Ind}_T^G (\bar{\chi} \otimes \xi^j)$, where each $\operatorname{Ind}_T^G (\bar{\chi} \otimes \xi^j)$ is irreducible by \Cref{thm:char}. Now $\operatorname{Ind}_T^G (\bar{\chi} \otimes \xi^j) = \operatorname{Ind}_T^G (\bar{\chi}) \otimes \epsilon^j$, where $\epsilon$ is the linear character with $\epsilon(H) = 1$ and $\epsilon(t) = \exp(2\pi i / m)$. So $\operatorname{Ind}_H^G \chi = \sum_{j=1}^{m/d} \operatorname{Ind}_T^G (\bar{\chi}) \otimes \epsilon^j$, and $m/d$ can be characterized as the smallest positive integer $k$ such that $\operatorname{Ind}_T^G (\bar{\chi}) \otimes \epsilon^k = \operatorname{Ind}_T^G (\bar{\chi})$.

Now suppose that we have some $\phi \in \operatorname{Irr} G$ such that $\langle \phi_H, \chi \rangle_H \neq 0$. By Frobenius reciprocity, it then follows that $\phi = \operatorname{Ind}_T^G(\bar{\chi}) \otimes \epsilon^j$ for some $j \in \{1,\dots,m/d\}$. Since all of the $\bar\chi \otimes \xi^j$ satisfy $(\bar\chi \otimes \xi^j)_H = \chi$, we can assume that $\phi = \operatorname{Ind}_T^G(\bar\chi)$. So $m/d$ is the smallest positive integer $k$ such that $\phi \otimes \epsilon^k = \phi$. Now $\langle \phi_H, \chi \rangle_H = \langle \bar\chi_H, \chi \rangle_H = 1$, so by Clifford's theorem we have
\[
    \phi_H = \sum_{j=0}^{d-1} \chi^{t^j}.
\]

Recall (see \cite[\S 2.7]{lehrer2009unitary}) that $G(m,m,n)$ is the subgroup of $G(m,1,n)$ generated by the reflections $\{t^{-1}s_1 t, s_1,s_2,\dots,s_{n-1}\}$. Moreover, $G(m,m,n)$ is a normal subgroup of $G(m,1,n)$, and the quotient $G(m,1,n)/G(m,m,n)$ is cyclic, generated by the image of $t$. Hence $G(m,1,n) \cong G(m,m,n) \rtimes \langle t \rangle$, so we can apply the above result to these groups. 

For an $m$-partition $\bm \lambda = (\lambda^{(0)}, \dots, \lambda^{(m-1)})$ of $n$, denote by $\pi(\bm\lambda)$ the cyclic permutation $(\lambda^{(m-1)}, \lambda^{(0)},\dots,\lambda^{(m-2)})$. Let $\langle \pi \rangle$ denote the cyclic group of order $m$, and let $s(\bm\lambda)$ be the size of the subgroup of $\langle \pi \rangle$ fixing $\bm\lambda$. Notice that $\chi_{\bm\lambda} \otimes \epsilon = \chi_{\pi(\bm\lambda)}$, so the smallest positive integer $k$ satisfying $\chi_{\bm\lambda} = \chi_{\bm\lambda} \otimes \epsilon^k = \chi_{\pi^k(\bm\lambda)}$ is $m / s(\bm\lambda)$. It follows that the restriction of $\chi_{\bm\lambda}$ to $G(m,m,n)$ is the sum of $s(\bm\lambda)$ distinct (conjugate) irreducible characters. Moreover, each irreducible character of $G(m,m,n)$ is in the restriction of exactly $m/s(\bm\lambda)$ irreducible characters of $G(m,1,n)$, which are equivalent under the action of $\langle \pi \rangle$ and have the same decomposition when restricted to $G(m,m,n)$. That is, the correspondence is as follows:
\begin{align*}
    \{\chi_{\bm\lambda}, \chi_{\pi(\bm\lambda)}, \dots, \chi_{\pi^{m/s(\bm\lambda) - 1}(\bm\lambda)}\} \in \operatorname{Irr}(G(m,1,n)) &\leftrightarrow \left\{\psi, \psi^{t},\dots, \psi^{t^{s(\bm\lambda)-1}}\right\} \in \operatorname{Irr}(G(m,m,n)) \\
    (\chi_{\pi^j(\bm\lambda)})_{G(m,m,n)} &= \psi + \psi^t + \cdots + \psi^{t^{s(\bm\lambda) - 1}} \\
    \chi_{\bm\lambda} + \chi_{\pi(\bm\lambda)} + \cdots + \chi_{\pi^{m/s(\bm\lambda) - 1}(\bm\lambda)} &= \operatorname{Ind}_{G(m,m,n)}^{G(m,1,n)} \psi^{t^j}.
\end{align*}

There is a similar description of the irreducible characters of $G(m,p,n)$, and there are are explicit models for the irreducible representations of $G(m,p,n)$ related to those shown above in the case $G(m,1,n)$. See \cite{ariki1995representation} and \cite[\S 2.4]{marin2010automorphisms}.

\begin{example}[Exterior powers and Galois twists]

Let $V$ be the reflection representation of $G(m,1,n)$. Then the reflection representation $\tilde{V}$ of $G(m,m,n)$ is the restriction of the action of $G(m,1,n)$ on $V$. Now consider the $G(m,m,n)$ representation $\Lambda^k \tilde{V}^{\sigma_p}$, where $p$ is coprime to $h = \max\{n,(n-1)m\}$ (the Coxeter number of $G(m,m,n)$) and $\sigma_p \in \operatorname{Gal}(\mathbb{Q}(\zeta_h)/\mathbb{Q})$. Assuming that $\Lambda^k V^{\sigma_p}$ is well-defined as a representation of $G(m,1,n)$, it is easy to see (by looking at the characters, for example) that the restriction of $\Lambda^k V^{\sigma_p}$ to $G(m,m,n)$ is precisely $\Lambda^k \tilde{V}^{\sigma_p}$.

For the groups $G(m,1,n)$, the field of definition is $\mathbb{Q}(\zeta_m)$. If $n > (n-1)m$, there are only two possibilities:
\begin{itemize}
    \item If $n = 1$, then $G(m,1,1)$ is a cyclic group and $G(m,m,1)$ is the trivial group. In this case $\operatorname{Gal}(\mathbb{Q}(\zeta_h)/\mathbb{Q})$ is trivial, so $\Lambda^k V^{\sigma_p}$ is clearly well-defined.

    \item If $n \geq 2$ and $m = 1$, then $G(m,1,n) = G(1,1,n)$ is not irreducible.
\end{itemize}
Otherwise $(n-1)m > n$, and $\Lambda^k V^{\sigma_p}$ is well defined because $\mathbb{Q}(\zeta_m) \subseteq \mathbb{Q}(\zeta_h)$.

So the $m$-partition (up to cyclic permutation) corresponding to $\Lambda^k \tilde{V}^{\sigma_p}$ is $\bm\lambda_{k,p} = (n-k,\varnothing,\dots,\varnothing, 1^k,\varnothing, \dots,\varnothing)$, where the $1^k$ is in the $p$th slot (mod $m$; indexing starts at 0). If $\bm\lambda_{k,p}$ has cyclic symmetry, then $n = 2$, $m$ is even, and $p = m/2$. In this case $(n-1)m > n$, so since $p$ must be relatively prime to $m$, we have $m = 2$. But the group $G(2,2,2)$ is not irreducible. So we have $s(\bm\lambda_{k,p}) = 1$ which agrees with the observation that the restriction of $\Lambda^k V^{\sigma_p}$ to $G(m,1,n)$ is precisely $\Lambda^k \tilde{V}^{\sigma_p}$.

Notice that $\Lambda^n \tilde{V}^{\sigma_p}$ does not depend on $p$ because $\bm\lambda_{n,p} = \pi^p(\bm\lambda_{n,1})$.
\end{example}

\subsection{Unipotent characters and Generic Degrees for \texorpdfstring{$G(m,m,n)$}{G(m,m,n)}}
\label{sec:m,m,n}

We now give Malle's construction \cite{malle1995unipotente} \cite{lasy2012traces} of the unipotent characters and generic degrees for the groups $G(m,m,n)$.

Define an equivalence relation on $m$-symbols as the symmetric transitive closure of the two operations: cyclic permutation of the $S_i$ in $S$, and \defn{simultaneous shift} of all of the $S_i$ given by
\[
    (\lambda_{i,1}, \dots, \lambda_{i,M_i}) \mapsto (0,\lambda_{i,1} + 1, \dots, \lambda_{i,M_i} + 1).
\]
Denote by $[S]$ the equivalence class of an $m$-symbol $S$. 

For $S$ an $m$-symbol with $I(S) \equiv 0 \mod m$, define its \defn{defect} by
\[
    \operatorname{def}(S) = \left(\frac{m-1}{2}I(S) - \sum_{i=0}^{m-1} i M_i\right) \mod m.
\]
Define $s(S)$ as the cardinality of the set $\{0 \leq i \leq e-1 : \pi^i(S) = S\}$, where $\pi$ cyclically permutes the $S_i$ in $S$. In the case $I(S) = mk$, the formula for rank takes the form
\[
    \operatorname{rk}(S) = \sum_{i=0}^{m-1} \sum_{\lambda \in S_i} \lambda - m \binom{k}{2}.
\]

\begin{definition}
\label{def:uniDeg2}
    The \defn{unipotent characters} of $G(m,m,n)$ are the elements of the multiset 
    \[
        \operatorname{Uch}(G(m,m,n)) := \{\text{equivalence classes $[S]$} : \operatorname{rk}(S) = n, \,\, I(S) \equiv \operatorname{def}(S) \equiv 0 \mod m\},
    \]
    where the multiplicity of $[S]$ in $\operatorname{Uch}(G(m,m,n))$ is $s(S)$ (none of these conditions/values depend on the choice of representative of $[S]$). 
    
    Families of unipotent characters will be defined as follows: If $s(S) = m$, we'll say that each copy of $[S]$ belongs to its own family consisting of only one element. If $[S]$, $[S']$ have representatives $S_1 \in [S]$ and $S_1' \in [S']$ such that the entries of $S_1$ and $S_1'$ coincide as multisets and $s(S) < m$, then $s(S') < m$, and we'll say that all copies of $[S]$ and $[S']$ belong to the same family.

    Let $S$ be any representative of $[S]$, then the \defn{generic degree} of $[S]$ is 
    \[
        \operatorname{Deg}_{[S]}(q) := \frac{m}{s(S)} \cdot \frac{(-1)^{\binom{m}{2}\binom{\ell}{2} + \gamma(S)} (q^n-1) \left(\prod_{i=1}^{n-1}(q^{mi}-1)\right)\cdot \prod_{i=0}^{m-1} \prod_{j=i}^{m-1} \prod_{\substack{(\lambda, \mu) \in S_i \times S_j \\ \mu < \lambda \text{ if } i = j}}(q^\lambda \zeta_m^i - q^\mu \zeta_m^j)}{\tau(m)^\ell q^{\binom{m(\ell-1)}{2} + \binom{m(\ell-2)}{2} + \cdots} \prod_{i=0}^{m-1} \Theta(S_i, q^m)}
    \]
    where
    \[
        \ell = \frac{I(S)}{m}, \qquad \gamma(S) = \frac{\operatorname{def}(S)}{m}(m\ell - 1),
    \]
    and $\tau(m)$ and $\Theta$ are as in \Cref{def:uniDeg1}. Malle's definition in \cite{malle1995unipotente} does not include the $(-1)^{\gamma(S)}$ term, in which case the generic degree is defined only up to sign. We have followed \cite{lasy2012traces} in this definition so that the generic degree does not depend on the choice of representative.
\end{definition}

Again there is an inclusion $\operatorname{Irr}(G(m,m,n)) \subset \operatorname{Uch}(G(m,m,n))$. Let $\bm\lambda = (\lambda^{(0)},\dots,\lambda^{(m-1)})$ be an $m$-partition of $n$. Let $\alpha_i = (0 \leq \alpha_{i,1} \leq \cdots \leq \alpha_{i,M})$ be the parts of $\lambda^{(i)}$ written in increasing order for $i \geq 0$ with 0's added if necessary. Define an $m$-symbol $\tilde{S}_{\bm\lambda} = (\tilde{S}_0,\dots,\tilde{S}_{m-1})$ by $\tilde{S}_{i,j} = \alpha_{i,j} + j - 1$. Then $\tilde{S}_{\bm\lambda}$ is a symbol of rank $n$ with $I(\tilde{S}_{\bm\lambda}) = mM$ and $\operatorname{def}(\tilde{S}_{\bm\lambda}) = 0$. The inclusion is then defined by mapping the $s(\bm\lambda)$ components of the restriction $(\chi_{\bm\lambda})_{G(m,m,n)}$ to the $s(\tilde{S}_{\bm\lambda}) = s(\bm\lambda)$ copies of $[\tilde{S}_{\bm\lambda}]$ in $\operatorname{Uch}(G(m,m,n))$.

Under this inclusion, the definition of the generic degree of an irreducible character given in \Cref{def:genDeg} coincides with the definition for unipotent characters.

\begin{proposition}[\cite{malle1995unipotente}, Proposition 5.5]
\label{prop:fegForm2}
    The fake degree of a component of the restriction $(\chi_{\bm\lambda})_{G(m,m,n)}$, with corresponding $m$-symbol $\tilde{S}_{\bm\lambda} = (\tilde{S}_0,\dots,\tilde{S}_{m-1})$, is equal to
    \[
        \frac{q^n - 1}{s(\tilde{S}_{\bm\lambda})}\cdot \left(\prod_{i=1}^{n-1} (q^{mi}-1)\right)\cdot \left(\prod_{i=0}^{m-1}\frac{\Delta(\tilde{S}_i,q^m)}{\Theta(\tilde{S}_i,q^m)} \right) \cdot \frac{\sum_{j=0}^{m-1} \prod_{i=1}^{m-1}\prod_{\lambda \in \tilde{S}_{(i+j) \mod m}}q^{(m-i)\lambda}} {q^{\binom{m(\ell-1)}{2} + \binom{m(\ell - 2)}{2} + \cdots}},
    \]
    where $\Delta$ is as in \Cref{prop:fegForm1} and $\ell$ and $\Theta$ are as in \Cref{def:uniDeg2}. We have again followed \cite[Proposition 5.4.13]{lasy2012traces} so that this formula agrees with the usual definition of fake degrees.
\end{proposition}

\begin{example}
\label{ex:galFeg2}
    The $m$-symbol corresponding to $\tilde{V}^{\sigma_p}$ is given by
    \[
        \tilde{S}_i = \begin{cases}
            (n-1) & \text{if } i = 0, \\
            (1) & \text{if } i = p \mod m, \\
            (0) & \text{otherwise}.
        \end{cases}
    \]
    So
    \[
        \Delta(\tilde{S}_i,q^m) = 1, \qquad \text{and} \qquad \Theta(\tilde{S}_i,q^m) = \begin{cases}
            \prod_{j=1}^{n-1} (q^{mj}-1) & \text{if } i = 0, \\
            (q^m - 1) & \text{if } i = p \mod m, \\
            1 & \text{otherwise}.
        \end{cases}
    \]
    Letting $\chi_{1,p}$ denote the character of $\tilde{V}^{\sigma_p}$, we then have
    \begin{align*}
        \operatorname{Feg}_{\chi_{1,p}}(q) &= \frac{q^n - 1}{q^m - 1} \sum_{j=0}^{m-1} q^{j(n-1) + ((j-p) \mod m)} \\
        &= \frac{q^n - 1}{q^m - 1}\left(\sum_{j=0}^{(p \mod m) - 1} q^{m + jn - (p \mod m))} + \sum_{j= p \mod m}^{m-1} q^{jn-(p\mod m)} \right) \\
        &= \frac{q^n - 1}{q^m - 1}\left((q^m - 1)\sum_{j=0}^{(p \mod m) - 1} q^{jn - (p\mod m)} + \sum_{j= 0}^{m - 1} q^{jn - (p \mod m)} \right) \\
        &= (q^n - 1) \sum_{j=0}^{(p \mod m) - 1} q^{jn - (p\mod m)} + q^{-(p \mod m)} \cdot \frac{q^n - 1}{q^m - 1} \cdot \frac{q^{mn}-1}{q^n - 1} \\
        &= q^{- (p \mod m)}\left(q^{n(p \mod m)} - 1 \right) + \sum_{j = 0}^{n-1} q^{jm-(p \mod m)} \\ 
        &= q^{(n-1)(p \mod m)} + \sum_{j = 1}^{n-1} q^{jm-(p \mod m)} = q^{p(n-1) \mod m(n-1)} + \sum_{j=1}^{n-1} q^{p(jm - 1) \mod m(n-1)} \\
        &= \sum_{j=1}^n q^{pe_i \mod h},
    \end{align*}
    which demonstrates \Cref{ex:galoisFeg}.
\end{example}

\begin{proof}[Proof of \Cref{thm:galDeg}]
    Let $\bm\lambda$ be the $m$-partition of $n$ corresponding to $\Lambda^k V^{\sigma_p}$, where $V$ is the reflection representation of $G(m,1,n)$ and $p$ is coprime to $(n-1)m$. The result is obvious for $k = 0$, so suppose that $k \geq 1$. Then $S_{\bm\lambda} = (S_0,\dots,S_{m-1})$, where
    \[
        S_i = \begin{cases}
            (0,1,\dots,k-1,n) & \text{if } i = 0, \\
            (1,2,\dots,k) & \text{if } i = p \mod m, \\
            (0,1,\dots,k-1) & \text{otherwise},
        \end{cases}
    \]
    and $\tilde{S}_{\bm\lambda} = (\tilde{S}_0,\dots,\tilde{S}_{m-1})$, where
    \[
        \tilde{S}_i = \begin{cases}
            (0,1,\dots,k-2,n-1) & \text{if } i = 0, \\
            (1,2,\dots,k-1,k) & \text{if } i = p \mod m, \\
            (0,1,\dots,k-2,k-1) & \text{otherwise}.
        \end{cases}
    \]
    When $k = 1$, $\tilde{S}_0 = (n-1)$. Comparing \Cref{def:uniDeg1} and \Cref{def:uniDeg2}, one sees that
    \begin{align*}
        \frac{\operatorname{Deg}_{[\tilde{S}_{\bm\lambda}]}(q)}{\operatorname{Deg}_{S_{\bm\lambda}}(q)} &= \frac{mq^{m\binom{k}{2}}(q^n-1)}{(q^{mn}-1)} \cdot \left( (q^{mn}-1)\prod_{i=1}^{k-1}(q^{mi}-1) \right) \cdot \frac{\left(\prod_{i=0}^{k-2} (q^{n-1} - q^i)\right)}{\left(\prod_{i=0}^{k-2} (q^{k-1} - q^i)\right)\left(\prod_{i=0}^{k-1} (q^{n} - q^i) \right)} \\
        & \qquad \left(\prod_{j = 1}^{m-1} \prod_{\mu = 0}^{k-1} \frac{(q^{n-1} - q^{\mu}\zeta_m^j)}{(q^{k-1} - q^\mu \zeta_m^j)(q^{n} - q^\mu \zeta_m^j)}\right) \cdot \frac{(q^{n-1} - q^{k}\zeta_m^p)(q^{k-1}-\zeta_m^p)(q^n - \zeta_m^p)}{(q^{k-1} - q^k \zeta_m^p)(q^{n} - q^k \zeta_m^p)(q^{n-1}-\zeta_m^p)} \\
        &= mq^{m\binom{k}{2}}(q^n-1)\left(\prod_{i=1}^{k-1}(q^{mi}-1)\right) \frac{1}{q^{(k-1) + \binom{k-1}{2}}(q^n - 1)\prod_{i=1}^{k-1}(q^i - 1)} \\
        & \qquad \left(\prod_{\mu = 0}^{k-2} \frac{(q^{(n-\mu-1)m}-1)(q^{k-\mu-1}-1)(q^{n-\mu} - 1)}{q^{\mu(m-1)}(q^{(k-\mu-1)m}-1)(q^{(n-\mu)m}-1)(q^{n-\mu-1}-1)}\right) \\
        & \qquad \left(\frac{(q^{n-k+1}-1)}{mq^{(k-1)(m-1)}(q^{(n-k+1)m}-1)} \prod_{j=1}^{m-1} (q^{n-k} - \zeta_m^j) \right) \cdot \frac{(q^{n-k-1} - \zeta_m^p)(q^{k-1}-\zeta_m^p)(q^n - \zeta_m^p)}{(q^{k-1} - q^k \zeta_m^p)(q^{n-k} - \zeta_m^p)(q^{n-1}-\zeta_m^p)} \\
        &= \left(\prod_{i=1}^{k-1}(q^{mi}-1)\right) \frac{q^{m\binom{k}{2}}}{q^{(k-1) + \binom{k-1}{2}}\prod_{i=1}^{k-1}(q^i - 1)} \\
        & \qquad \left(\frac{1}{q^{(m-1)\binom{k-1}{2}}} \cdot \frac{\left(\prod_{\mu = n-k+1}^{n-1}(q^{m\mu} - 1) \right)\left(\prod_{\mu = 1}^{k-1} (q^{\mu}-1)\right) \left(\prod_{\mu = n - k+2}^{n} (q^{\mu} - 1)\right)}{\left(\prod_{\mu = 1}^{k-1}(q^{m\mu}-1)\right)\left(\prod_{\mu = n-k+2}^n (q^{m\mu}-1)\right)\left(\prod_{\mu = n-k+1}^{n-1}(q^{\mu}-1)\right)}\right) \\
        & \qquad \left(\frac{(q^{n-k+1}-1)}{q^{(k-1)(m-1)}(q^{(n-k+1)m}-1)} \prod_{j=1}^{m-1} (q^{n-k} - \zeta_m^j) \right) \cdot \frac{(q^{n-k-1} - \zeta_m^p)(q^{k-1}-\zeta_m^p)(q^n - \zeta_m^p)}{(q^{k-1} - q^k \zeta_m^p)(q^{n-k} - \zeta_m^p)(q^{n-1}-\zeta_m^p)} \\
        &= \frac{q^{m\binom{k}{2}}\prod_{j=1}^{m-1} (q^{n-k} - \zeta_m^j)}{q^{m(k-1) + m\binom{k-1}{2}}} \left(\frac{(q^{m(n-k+1)} - 1)(q^{n} - 1)}{(q^{mn}-1)(q^{n-k+1}-1)}\right) \left(\frac{(q^{n-k+1}-1)}{(q^{(n-k+1)m}-1)}  \right)\\
        & \qquad \frac{(q^{n-k-1} - \zeta_m^p)(q^{k-1}-\zeta_m^p)(q^n - \zeta_m^p)}{(q^{k-1} - q^k \zeta_m^p)(q^{n-k} - \zeta_m^p)(q^{n-1}-\zeta_m^p)} \\
        &= \frac{\alpha(q^{n} - 1)(q^{n-k-1} - \zeta_m^p)(q^{k-1}-\zeta_m^p)(q^n - \zeta_m^p)}{(q^{mn}-1)(q^{k-1} - q^k \zeta_m^p)(q^{n-k} - \zeta_m^p)(q^{n-1}-\zeta_m^p)} ,
    \end{align*}
    where
    \[
        \alpha = \prod_{j=1}^{m-1} (q^{n-k} - \zeta_m^j) = \begin{cases}
            m & \text{if } k = n, \\
            \frac{q^{m(n-k)}-1}{q^{n-k}-1} & \text{otherwise}.
        \end{cases}
    \]
    We will first handle the case $k = n$. By \Cref{eqn:deg}, 
    \[
        \operatorname{Deg}_{[\tilde{S}_{\bm\lambda}]}(q) = \frac{q^{n + m\binom{n}{2}}}{q^{n}} = q^{m\binom{n}{2}} \implies \operatorname{Deg}_{[\tilde{S}_{\bm\lambda}]}(\zeta^p_{m(n-1)}) = (-1)^{np}.
    \]
    If $n$ is odd, then $p$ must be odd because it is coprime to $n-1$, so then $np$ is odd. If $n$ is even, then certainly $np$ is even. Thus $\operatorname{Deg}_{[\tilde{S}_{\bm\lambda}]}(\zeta^p_{m(n-1)}) = (-1)^n$.

    Now for $0 < k < n$, by \Cref{eqn:deg},
    \begin{align*}
        \operatorname{Deg}_{[\tilde{S}_{\bm\lambda}]}(q) &= \frac{q^{k + m\binom{k}{2}}(q^{n} - 1)(q^{n-k-1} - \zeta_m^p)(q^{k-1}-\zeta_m^p)\prod_{i=k}^{n-1}(q^{mi}-1)}{m(q^{n-k}-1)(q^{k-1} - q^k \zeta_m^p)(q^{n-1}-\zeta_m^p)(q^k-1)\prod_{j=1}^{n-k-1}(q^{mj}-1)} \\
        &= \frac{q^{k + m\binom{k}{2}}(q^{n} - 1)(q^{n-k-1} - \zeta_m^p)(q^{k-1}-\zeta_m^p)\prod_{i=k}^{n-2}(q^{mi}-1)}{m(q^{n-k}-1)(q^{k-1} - q^k \zeta_m^p)(q^k-1)\prod_{j=1}^{n-k-1}(q^{mj}-1)} \prod_{\substack{j=0 \\ j \neq p \mod m}}^{m-1} (q^{n-1} - \zeta_m^j).
    \end{align*}
    Hence,
    \begin{align*}
        \operatorname{Deg}_{[\tilde{S}_{\bm\lambda}]}(\zeta_{m(n-1)}^p) &= \frac{\zeta_{m(n-1)}^{pk} \zeta_{n-1}^{p\binom{k}{2}}(\zeta_{m(n-1)}^{pn} - 1)(\zeta_{m}^{p}\zeta_{m(n-1)}^{-kp} - \zeta_m^p)(\zeta_{m(n-1)}^{p(k-1)}-\zeta_m^p)\prod_{i=k}^{n-2}(\zeta_{n-1}^{pi}-1)}{m(\zeta_{m(n-1)}^{p(n-k)}-1)(\zeta_{m(n-1)}^{p(k-1)} - \zeta_{m(n-1)}^{pk} \zeta_m^p)(\zeta_{m(n-1)}^{pk}-1)\prod_{j=1}^{n-k-1}(\zeta_{n-1}^{pj}-1)} \\
        & \qquad \cdot \prod_{\substack{j=0 \\ j \neq p \mod m}}^{m-1} (\zeta_m^{p} - \zeta_m^j) \\
        &= \frac{\zeta_{m(n-1)}^{pk} \zeta_{n-1}^{p\binom{k}{2}}(\zeta_{m(n-1)}^{pn} - 1)(\zeta_m^p)(\zeta_{m(n-1)}^{-pk} - 1)(-\zeta_{m(n-1)}^{p(k-1)})(\zeta_{m(n-1)}^{p(n-k)}-1)}{m(\zeta_{m(n-1)}^{p(n-k)}-1)(-\zeta_{m(n-1)}^{p(k-1)})(\zeta_{m(n-1)}^{pn} - 1)(\zeta_{m(n-1)}^{pk}-1)} \\ 
        & \qquad \frac{\prod_{i=k}^{n-2}(\zeta_{n-1}^{pi}-1)}{\prod_{j=1}^{n-k-1}(-\zeta_{n-1}^{pj})(\zeta_{n-1}^{p(n-j-1)}-1)}(\zeta_m^{-p}m) \\
        &= \frac{\zeta_{m(n-1)}^{pk} \zeta_{n-1}^{p\binom{k}{2}}(\zeta_{m(n-1)}^{-pk} - 1)}{(-1)^{n-k-1}\zeta_{n-1}^{p\binom{n-k}{2}}(\zeta_{m(n-1)}^{pk}-1)} = (-1)^{p(2k-n) - (n-k)} = (-1)^{np}(-1)^{n}(-1)^k = (-1)^k.
    \end{align*}

    Now suppose that $\bm\lambda$ is any $m$-partition of $n$ such that $\operatorname{Deg}_{[\tilde{S}_{\bm\lambda}]}(\zeta_{m(n-1)}^p) \neq 0$. Choose $\tilde{S}_{\bm\lambda} = (\tilde{S}_0,\dots,\tilde{S}_{m-1})$ so that $I(\tilde{S}_{\bm\lambda})$ is minimal. Let $\tilde{k} := I(\tilde{S}_{\bm\lambda})/m = \max \{\text{number of parts in } \lambda^{(i)} : i = 0,\dots,n-1\}$. Since $\operatorname{Deg}_{[\tilde{S}_{\bm\lambda}]}(q)$ has the factor $(q^{m(n-1)}-1)$ in the numerator, at least one of the $\Theta(\tilde{S}_i,\zeta_{n-1}^p)$ is zero. This implies that there is at least one $\mu \in \bigcup \tilde{S}_i$ with $\mu \geq n-1$. There are four cases to consider:
    \begin{enumerate}
        \item Suppose that there is some $j > 0$ such that (up to cyclic permutation)
        \[
            \lambda^{(t)} = \begin{cases}
                n-\tilde{k} & \text{if } t = 0 \\
                1^{\tilde{k}} & \text{if } t = j \\
                \varnothing & \text{otherwise}.
            \end{cases}
        \]
        Then the $m$-symbol is
        \[
        \tilde{S}_t = \begin{cases}
            (0,1,\dots,\tilde{k}-2,n-1) & \text{if } t = 0, \\
            (1,2,\dots,\tilde{k}-1,\tilde{k}) & \text{if } t = j, \\
            (0,1,\dots,\tilde{k}-2,\tilde{k}-1) & \text{otherwise}.
        \end{cases}
        \]
        If $\tilde{k} = n$, then $\bm\lambda$ corresponds to $\Lambda^{n} \tilde{V}^{\sigma_p}$. So suppose that $\tilde{k} \leq n - 1$. If $j = p \mod m$, then $\bm\lambda$ corresponds to $\Lambda^{\tilde{k}} \tilde{V}^{\sigma_p}$. We will show that this is the only possibility. Suppose for the sake of contradiction that $j \neq p \mod m$. Then $m \geq 3$ and there are two cases to consider:
        \begin{itemize}
            \item If $\tilde{k} = n-1$, then there are exactly two $\mu \in \bigcup \tilde{S}_i$ equal to $n-1$. Since $p$ is relatively prime to $m$, we have that $p \neq 2p \,\, (\mod m)$, so the numerator of $\operatorname{Deg}_{[\tilde{S}_{\bm\lambda}]}(q)$ has the factor $\pm(q^{n-1}\zeta_m^p - \zeta_m^{2p})$. Additionally, since we are assuming $j \neq p$, the numerator of $\operatorname{Deg}_{[\tilde{S}_{\bm\lambda}]}(q)$ has the factor $(q^{n-1} - \zeta_m^p)$. But it then follows that $\operatorname{Deg}_{[\tilde{S}_{\bm\lambda}]}(\zeta_{m(n-1)}^p) = 0$, a contradiction.
    
            \item If $\tilde{k} < n-1$, then there is only one $\mu \in \bigcup \tilde{S}_i$ equal to $n-1$. Since we are assuming $j \neq p$, the numerator of $\operatorname{Deg}_{[\tilde{S}_{\bm\lambda}]}(q)$ has the factor $(q^{n-1} - \zeta_m^p)$. So $\operatorname{Deg}_{[\tilde{S}_{\bm\lambda}]}(\zeta_{m(n-1)}^p) = 0$, a contradiction.
        \end{itemize}
        
        \item Suppose that $1 \leq \tilde{k} \leq n-1$ and (up to cyclic permutation)
        \[
            \lambda^{(t)} = \begin{cases}
                (n-\tilde{k} + 1, 1^{\tilde{k}-1}) & \text{if } t = 0 \\
                \varnothing & \text{otherwise}.
            \end{cases}
        \]
        Then the $m$-symbol is
        \[
            \tilde{S}_{\bm\lambda} = \begin{pmatrix}
                1 & 2 & \cdots & \tilde{k}-1 & n \\
                0 & 1 & \cdots & \tilde{k}-2 & \tilde{k}-1 \\
                \vdots & \vdots & \ddots & \vdots & \vdots \\
                0 & 1 & \cdots & \tilde{k}-2 & \tilde{k}-1
            \end{pmatrix}
        \]
        If $\tilde{k} = 1$, then $\bm\lambda$ corresponds to $\Lambda^{0} \tilde{V}^{\sigma_p}$. So suppose that $2 \leq \tilde{k} \leq n-1$. The numerator of $\operatorname{Deg}_{[\tilde{S}_{\bm\lambda}]}(q)$ has the factor $(q^n - q\zeta_m^p)$. So $\operatorname{Deg}_{[\tilde{S}_{\bm\lambda}]}(\zeta_{m(n-1)}^p) = 0$, a contradiction.

        \item Suppose that $2 \leq \tilde{k} \leq n - 2$ and (up to cyclic permutation)
        \[
            \lambda^{(t)} = \begin{cases}
                (n-\tilde{k},2, 1^{\tilde{k}-2}) & \text{if } t = 0 \\
                \varnothing & \text{otherwise}.
            \end{cases}
        \]
        Then the $m$-symbol is
        \[
            \tilde{S}_{\bm\lambda} = \begin{pmatrix}
                1 & 2 & \cdots & \tilde{k}-2 & \tilde{k} & n - 1 \\
                0 & 1 & \cdots & \tilde{k}-3 & \tilde{k}-2 & \tilde{k}-1 \\
                \vdots & \vdots & \ddots & \vdots & \vdots & \vdots \\
                0 & 1 & \cdots & \tilde{k}-3 & \tilde{k}-2 & \tilde{k}-1
            \end{pmatrix}
        \]
        The numerator of $\operatorname{Deg}_{[\tilde{S}_{\bm\lambda}]}(q)$ has the factor $(q^{n-1} - \zeta_m^p)$. So $\operatorname{Deg}_{[\tilde{S}_{\bm\lambda}]}(\zeta_{m(n-1)}^p) = 0$, a contradiction.

        \item Suppose that $2 \leq \tilde{k} \leq n-2$ and there is some $j > 0$ such that (up to cyclic permutation)
        \[
            \lambda^{(t)} = \begin{cases}
                (n-\tilde{k}, 1^{\tilde{k}-1}) & \text{if } t = 0 \\
                1 & \text{if } t = j \\
                \varnothing & \text{otherwise}.
            \end{cases}
        \]
        Then the $m$-symbol is
        \[
            \tilde{S}_{t} = \begin{cases}
                (1,2,\dots,\tilde{k}-1, n-1) & \text{if } t = 0 \\
                (0,1,\dots,\tilde{k}-2, \tilde{k}) & \text{if } t = j \\
                (0,1,\dots,\tilde{k}-2, \tilde{k}-1) & \text{otherwise}.
            \end{cases}
        \]
        The numerator of $\operatorname{Deg}_{[\tilde{S}_{\bm\lambda}]}(q)$ has the factor $(q^{n-1} - \zeta_m^p)$. So $\operatorname{Deg}_{[\tilde{S}_{\bm\lambda}]}(\zeta_{m(n-1)}^p) = 0$, a contradiction.
    \end{enumerate}
    
    To see that these are the only possible cases, choose $\ell$ such that the number of parts in $\lambda^{(\ell)}$ is $\tilde{k}$. Suppose that there is some $\mu \in \tilde{S}_i$ with $\mu \geq n-1$ and $i \neq \ell$. Then $\mu \leq (n - \tilde{k}) - \tilde{k} + 1 = n-1$ with equality only if $\bm\lambda$ is in case (1).

    Suppose now that there is some $\mu \in \tilde{S}_{\ell}$ with $\mu \geq n-1$. 
    \begin{itemize}
        \item If $\tilde{S}_{\ell,\tilde{k}} = n$, then $\lambda^{(\ell)}_{1} = n - \tilde{k} + 1$. This is case (1) if $\tilde{k} = n$, and it is case (2) otherwise.

        \item If $\tilde{S}_{\ell,\tilde{k}} = n-1$, then $\lambda^{(\ell)}_{1} = n - \tilde{k}$. This is case (1) if $\tilde{k} = 1$ or $n-1$. Otherwise it is case (3) or (4).
    \end{itemize}

    Thus $\operatorname{Deg}_{[\tilde{S}_{\bm\lambda}]}(\zeta_{m(n-1)}^p) = 0$ if $\bm\lambda$ does not correspond to an exterior power of $\tilde{V}^{\sigma_p}$.
\end{proof}

\begin{question}
    Let $W$ be a well-generated irreducible complex reflection group with Coxeter number $h$. Is it true that $S_{\sigma_p(\chi)}(\sigma_p(q)) = \sigma_p(S_{\chi}(q))$ for all $\sigma_p \in \operatorname{Gal}(\mathbb{Q}(\zeta_h) / \mathbb{Q})$ and $q \in \mathbb{Q}(\zeta_h)$?
\end{question}

\subsection{Families of unipotent characters in general} 

Let $\mathcal{O}$ be a commutative ring with a unit, and let $A$ be an $\mathcal{O}$-algebra. An \defn{idempotent} in $A$ is an element satisfying $e^2 = e$. Idempotents $e_1$ and $e_2$ are \defn{orthogonal} if $e_1e_2 = e_2e_1 = 0$. An idempotent is \defn{central} if it is contained in the center of $A$. An idempotent is \defn{primitive} if it is non-zero and is not equal to the sum of two non-zero orthgonal idempotents. 

Let $e$ be a central primitive idempotent. The two-sided ideal $Ae$ inherits an algebra structure, and we call the algebra $Ae$ a \defn{block} of $A$.

Let $W$ be an irreducible spetsial complex reflection group. Define the \defn{Rouquier ring} $\mathcal{R}_W(y)$ to be the $\mathbb{Z}_{k_W}$-subalgebra of $k_W(y)$ given by
\[
    \mathcal{R}_W(y) := \mathbb{Z}_{k_W}[y,y^{-1},(y^n - 1)^{-1}_{n \geq 1}].
\]
Then the \defn{Rouquier blocks} of $\mathcal{H}_q(W)$ are the blocks of the algebra $\mathcal{R}_W(y)\mathcal{H}_q(W)$, where $y^{|\mu(k_W)|} = q$. See \cite{chlouveraki2009blocks} for much more on these blocks.

To each $\chi \in \operatorname{Irr}(W)$ we can associate a central primitive idempotent $e_\chi$ in $k_W(y)\mathcal{H}_q(W)$ given by
\[
    e_\chi := \frac{1}{S_\chi(q)} \sum_{b \in \mathcal{B}} \chi_q(b)b^\vee,
\]
where $\mathcal{B}$ is a basis of $\mathcal{H}_q(W)$ adapted to the Wedderburn decomposition, and the $b^\vee$ form the dual basis with respect to $\tau_q$ (see \cite[Proposition 7.2.7]{geck2000characters} and \cite[Proposition 2.2.12]{chlouveraki2009blocks}). 

By \cite[Theorem 2.1.6]{chlouveraki2009blocks}, there exists a unique partition $\operatorname{RB(W)}$ of $\operatorname{Irr}(W)$ such that 
\begin{itemize}
    \item for each $B \in \operatorname{RB}(W)$, the element $e_B := \sum_{\chi \in B} e_\chi$ is a central primitive idempotent in $\mathcal{R}_W(y)\mathcal{H}_q(W)$,
    \item $1 = \sum_{B \in \operatorname{RB}(W)} e_B$ and for every central idempotent $e$ of $\mathcal{R}_W(y)\mathcal{H}_q(W)$ there exists a subset $\operatorname{RB}(W,e)$ of $\operatorname{RB}(W)$ such that $e = \sum_{B \in \operatorname{RB}(W,e)} e_B$.
\end{itemize}

We then say that two characters $\chi,\phi \in \operatorname{Irr}(W)$ belong to the same Rouquier block of $\mathcal{H}_q(W)$ if they belong to the same element of $\operatorname{RB}(W)$. We then have

\begin{theorem}
\label{thm:aA}
    If $\chi$ and $\phi$ belong to the same Rouquier block of the spetsial Hecke algebra, then $a_\chi = a_\phi$ and $A_\chi = A_\phi$.
\end{theorem}

See the remark at the end of chapter 4 of \cite{chlouveraki2009blocks} for notes on the proof and a more general statement. For Weyl groups, these Rouquier blocks are precisely the families of $\operatorname{Irr}(W)$. For the spetsial groups $G(m,1,n)$ and $G(m,m,n)$, the families defined above for unipotent characters recover the Rouquier blocks via the inclusion $\operatorname{Irr}(W) \hookrightarrow \operatorname{Uch}(W)$.

For $W$ an exceptional irreducible spetsial complex reflection group, there is a set $\operatorname{Uch}(\mathbb{G})$ defined in \cite{broue2014split} for the corresponding split spets. There is also a principal series $\operatorname{Uch}(\mathbb{G},1)$ with bijection $\operatorname{Irr}(W) \to \operatorname{Uch}(\mathbb{G},1)$. Moreover, by \cite[Axiom 5.20]{broue2014split} there is a partition of $\operatorname{Uch}(\mathbb{G})$ into families which recovers the Rouquier blocks of $\mathcal{H}_q(W)$ when restricted to the principal series.

\section{Lusztig's Fourier transform}
\label{sec:LFT}

The following lemma is our motivation for this section.

\begin{lemma}
\label{lem:transform}
    Suppose that there exists a pairing $\{-,-\}_W : \operatorname{Irr}(W) \times \operatorname{Irr}(W) \to \mathbb{C}$ satisfying 
    \begin{enumerate}
        \item[(T1)] For all $\chi \in \operatorname{Irr}(W)$, we have
        \[
            \operatorname{Deg}_\chi(q) = \sum_{\phi \in \operatorname{Irr}(W)} \{\chi, \phi\}_W \operatorname{Feg}_\phi(q).
        \]
        
        \item[(T2)] For all $\chi, \phi \in \operatorname{Irr}(W)$, we have $\{\chi, \phi\}_W = \{\phi, \chi\}_W$.
    
        \item[(T3)] For all $\chi, \phi \in \operatorname{Irr}(W)$ with $\{\chi, \phi\}_W \neq 0$, we have $h_\chi = h_\phi$.
    \end{enumerate}

    Then
    \[
        \sum_{\chi \in \operatorname{Irr}(W)} q_1^{f(\chi)}\operatorname{Feg}_\chi(q_2)\operatorname{Deg}_\chi(q_3) = \sum_{\chi \in \operatorname{Irr}(W)} q_1^{f(\chi)}\operatorname{Feg}_\chi(q_3)\operatorname{Deg}_\chi(q_2),
    \]
    if $f$ satisfies $h_\chi = h_\phi \implies f(\chi) = f(\phi)$.
\end{lemma}
\begin{proof}
    This follows from a double-summation argument.
\end{proof}

\begin{conjecture}
\label{conj:transform}
    For all irreducible spetsial complex reflection groups $W$, there exists a pairing $\{-,-\}_W$ satisfying (T1), (T2), and (T3).
\end{conjecture}

This conjecture has been proven for $W$ a finite Coxeter group and $W = G(m,1,n)$.

\subsection{A general construction}

This construction is due to Lusztig, but we will follow \cite[4.2.9]{geck2020character}. Let $G$ be a finite group and $\mathscr{M}(G)$ the set of pairs $(x, \sigma) \in G \times \operatorname{Irr}(C_G(x))$ modulo the relation $(x,\sigma) \sim ({}^g\!x, \sigma^g)$, where $\sigma^g({}^g\!y) := \sigma(y)$ for $y \in C_G(x)$. Lusztig defined a pairing $\{ \cdot ,\cdot \}_G : \mathscr{M}(G) \times \mathscr{M}(G) \to \mathbb{C}$ by
\[
    \{(x,\sigma),(y,\tau)\} := \frac{1}{|C_G(x)| |C_G(y)|} \sum_{\substack{g \in G \\ xgyg^{-1} = gyg^{-1}x}}\sigma(gyg^{-1})\tau(g^{-1}x^{-1}g).
\]
Let $\mathbf{M}_G$ be the operator on the space of functions $\mathscr{M}(G) \to \mathbb{C}$ defined by
\[
    (\mathbf{M}_G f)(x,\sigma) := \sum_{(y,\tau) \in \mathcal{M}(G)} \{(x,\sigma),(y,\tau)\} f(y,\tau).
\]
This is called the \defn{non-abelian Fourier transform} associated to $G$.

\subsection{The Fourier transform for Weyl groups} 

Suppose we are in the setting of \Cref{sec:FGOLT}. That is, $\mathbf{G}$ is a simple connected reductive group over $\overline{\mathbb{F}}_p$ with connected center which has Weyl group $W$, and $F : \mathbf{G} \to \mathbf{G}$ is a Frobenius map with respect to some $\mathbb{F}_q$-rational structure which acts trivially on $W$. Also fix some maximally split torus $\mathbf{T}_0$.

For each family $\mathscr{U} \subseteq \operatorname{Uch}(\mathbf{G}^F)$ with corresponding family $\mathscr{F} \subseteq \operatorname{Irr}(W)$, Lusztig \cite[\S 4]{lusztig1984characters} defines the following data:
\begin{itemize}
    \item A finite group $\mathscr{G}_{\mathscr{U}}$
    \item A bijection $\mathscr{U} \to \mathscr{M}(\mathscr{G}_{\mathscr{U}})$, $\rho \mapsto x_\rho$
    \item A map $\Delta : \mathscr{M}(\mathscr{G}_{\mathscr{U}}) \to \{\pm 1\}$ such that 
    \[
        \langle \rho, R_{\phi}\rangle = \Delta(x_\rho)\{x_\rho, x_\phi\}
    \]
    for all $\rho \in \mathscr{U}$ and $\phi \in \mathscr{F}$, where $x_\phi$ is defined via the inclusion $\operatorname{Irr}(W) \hookrightarrow \mathcal{E}(\mathbf{G}^F, (\mathbf{T}_0^F,1))$ and $\{ \cdot , \cdot \}$ is the non-abelian Fourier transform associated to $\mathscr{G}_{\mathscr{U}}$. In our context $\Delta(x_\rho) = 1$ except for some non-principal series characters in types $E_7$ and $E_8$.
\end{itemize}

The class functions 
\[
    R_\psi := \sum_{\rho \in \mathscr{U}} \Delta(x_\rho) \{ x_\rho, x_\psi\} \rho
\]
for $x \in \mathscr{M}(\mathscr{G}_{\mathscr{U}})$ are called \defn{unipotent almost characters} of $\mathbf{G}^F$. These have a geometric interpretation as the characteristic functions of certain $F$-stable character sheaves on $\mathbf{G}$ \cite[\S 7]{geck2017first}. When $\psi \in \operatorname{Irr}(W)$, this agrees with the unipotent uniform almost character $R_\psi$. If $\psi \not\in \operatorname{Irr}(W)$, then $R_\psi(1) = 0$. So the ``fake degree'' of a non-principal series unipotent character $\psi$ is 0.

Denote by $\operatorname{Uch}(\mathscr{F}) := \mathscr{U}$ and $\operatorname{Alm}(\mathscr{F}) := \{R_\psi : \psi \in \operatorname{Uch}(\mathscr{F})\}$. Then the corresponding \defn{non-abelian Fourier transform matrix} $\mathbf{M}(\mathscr{F}) := (\{x_\rho, x_\psi\})$ has columns indexed by $\rho \in \operatorname{Uch}(\mathscr{F})$ and rows indexed by $R_\psi \in \operatorname{Alm}(\mathscr{F})$. These matrices satisfy several important properties which form an axiomatic framework for Fourier matrices.

\subsection{The Fourier transform for finite Coxeter groups}

This axiomatic framework is presented in \cite[Theorem 6.9]{geck2003fourier}; the authors also describe Fourier matrices for non-Weyl Coxeter groups and show that they satisfy these axioms. 

Let $W$ be a finite Coxeter group, and let $\sigma : W \to W$ be an automorphism which leaves a set of simple reflections invariant. For Weyl groups, $\sigma$ will be the action of $F$ on $W$. For $\mathscr{F} \subseteq \operatorname{Irr}(W)$ a $\sigma$-invariant family, let $\mathbf{M}(\mathscr{F}) = (a_{\lambda, \mu})$ be the corresponding Fourier matrix with columns indexed by $\mu \in \operatorname{Uch}(\mathscr{F})$ and rows indexed by $\lambda \in \operatorname{Alm}(\mathscr{F})$. For Weyl groups with $\sigma$ trivial, these objects are as we defined above; for Weyl groups with $\sigma$ not trivial, the construction is similar \cite[\S 4.2]{geck2020character}. For non-Weyl Coxeter groups, these objects are described in \cite{geck2003fourier} along with an inclusion $\mathscr{F} \hookrightarrow \operatorname{Uch}(\mathscr{F})$. The unipotent degree of an element of $\operatorname{Uch}(\mathscr{F})$ is a polynomial in $q$ which recovers the generic degree of characters of $W$ via the inclusion.

A \defn{twisting operator} $t_1^*$ on the almost characters $\operatorname{Alm}(\mathscr{F})$ is defined for non-Weyl Coxeter groups in \cite{geck2003fourier}. For Weyl groups, this operator is defined on all class functions on $\mathbf{G}^F$ by
\[
    t_1^*(f)(x) := f(yxy^{-1}),
\]
for $x \in \mathbf{G}^F$, where $y \in \mathbf{G}$ satisfies $x = y^{-1}F(y)$ \cite{asai1984unipotent}. The almost characters are eigenvectors of $t_1^*$, and we denote by $\mathbf{F}_1$ the diagonal matrix of the corresponding eigenvalues.

To each unipotent character $\rho$ can be associated a root of unity $\omega_\rho$ called the \defn{Frobenius eigenvalue} of $\rho$. For Weyl groups, this can be described as follows \cite[4.2.21]{geck2020character}: 

Fix an $F$-stable Borel $\mathbf{B}_0$ containing $\mathbf{T}_0$. For any $w \in W$, define the corresponding \defn{Deligne-Lusztig variety}
\[
    \mathbf{X}_w := \{g\mathbf{B}_0 \in G / \mathbf{B}_0 : g^{-1}F(g) \in \mathbf{B}_0 w \mathbf{B}_0\}.
\]
Let $\delta > 0$ be the smallest positive integer such that $\sigma^\delta$ acts trivially on $W$. There exists some $w \in W$, $i \geq 0$, and $\mu \in \overline{\mathbb{Q}}_\ell^\times$ such that $\rho$ is in the character of the generalized $\mu$-eigenspace of $F^\delta$ on the $\ell$-adic cohomology group $H_c^i(\mathbf{X}_w,\overline{\mathbb{Q}}_\ell)$. Moreover, $\mu$ is uniquely determined by $\rho$ (independently of $w$ and $i$) up to an integral power of $q^\delta$. There is then a well-defined root of unity $\omega_\rho \in \overline{\mathbb{Q}}_\ell$, element $\lambda_\rho \in \{1,q^{\delta/2}\}$, and integer $s \geq 0$ such that $\mu = \omega_\rho \lambda_\rho q^{s\delta}$.

Denote by $\mathbf{F}_2$ the diagonal matrix of eigenvalues of Frobenius. When $\sigma$ is trivial, $\mathbf{F}_1 = \mathbf{F}_2$, which will just be denoted by $\mathbf{F}$. Let $\Delta$ be the permutation matrix describing the complex conjugation of the almost characters.

\begin{theorem}[\cite{geck2003fourier}, Theorem 6.9] For $W$ a finite Coxeter group, the following axioms are satisfied for each $\mathbf{M} = \mathbf{M}(\mathscr{F})$.
    \begin{enumerate}
        \item[(F1)] $\mathbf{M}$ transforms the vector of unipotent degrees to the vector of fake degrees (extended by zeros).

        \item[(F2)] All entries of $\mathbf{M}$ are real.

        \item[(F3)] $\mathbf{M} \cdot \mathbf{M}^{\operatorname{tr}} = \mathbf{M}^{\operatorname{tr}} \cdot \mathbf{M} = 1$.

        \item[(F4)] Let $\lambda_0$ be the row index of the almost character corresponding to the special character in $\mathscr{F}$. Then all entries in this row are non-zero.

        \item[(F5)] The structure constants
        \[
            a_{\lambda \mu}^\nu := \sum_{\kappa \in \operatorname{Uch}(\mathscr{F})} \frac{a_{\lambda \kappa} a_{\mu \kappa} a_{\nu \kappa}}{a_{\lambda_0 \kappa}}
        \]
        are rational integers for all $\lambda, \mu, \nu \in \operatorname{Alm}(\mathscr{F})$.

        \item[(F6)] If $\sigma = 1$ then $\mathbf{M} = \mathbf{M}^{\operatorname{tr}}$, $\Delta \cdot \mathbf{M} = \mathbf{M} \cdot \Delta$, and $(\mathbf{F} \cdot \Delta \cdot \mathbf{M})^3 = 1$.

        \item[(F6')] If $\sigma \neq 1$ and $\sigma^2 = 1$, then $(\mathbf{F}_2 \cdot \mathbf{M}^{\operatorname{tr}} \cdot \mathbf{F}_1^{-1} \cdot \mathbf{M})^2 = 1$.

        \item[(F6'')] If $\sigma \neq 1$ and $\sigma^3 = 1$, then $(\mathbf{F}_2 \cdot \mathbf{M}^{\operatorname{tr}} \cdot \mathbf{F}_1^{-1} \cdot \mathbf{M})^3 = 1$.
    \end{enumerate}
\end{theorem}

We can now define the pairing $\{-,-\}_W$ for finite Coxeter groups by 
\[
    \{\chi, \phi\}_W = \begin{cases}
        \mathbf{M}_{\chi,\phi} & \text{if $\chi$ and $\phi$ are in the same family} \\
        0 & \text{otherwise}
    \end{cases}
\]
where the Fourier matrices are taken for $\sigma = 1$. That is, the pairing corresponds to the block diagonal matrix consisting of Fourier matrices on families restricted to the principal series. Property (T1) follows from (F1), (F3), and (F6). Property (T2) follows from (F6), and property (T3) follows from \Cref{thm:aA}.

\subsection{The Fourier transform for \texorpdfstring{$G(m,1,n)$}{G(m,1,n)}}

This construction follows \cite[\S 4A]{malle1995unipotente}. Let $\mathscr{F}$ be a family of unipotent characters of $G(m,1,n)$, and fix some $S \in \mathscr{F}$. Let $Y$ be a totally ordered set with $m\ell + 1$ elements, $m > 0$, and let $\Psi = \Psi(Y)$ be the set of functions
\[
   \psi : Y \to \{0,1,\dots,m-1\} \quad \text{such that} \quad \sum_{y \in Y} \psi(y) \equiv \ell \binom{m}{2} \mod m.
\]

Let $\pi : Y \to \mathbb{Z}_{\geq 0}$ be a function such that $|\pi^{-1}(k)| = \#\{i : k \in S_i\}$. This function gives rise to an equivalence relation on $\Psi$ with $\phi \sim_\pi \psi$ if $\pi \circ \phi^{-1}(i) = \pi \circ \psi^{-1}(i)$ for $0 \leq i \leq m-1$. We'll say that $\psi \in \Psi$ is $\pi$-admissible if $\pi(y) = \pi(y')$ and $\psi(y) = \psi(y')$ implies $y = y'$. Denote by $[\psi]_\pi$ the equivalence class under this relation. The map from equivalence classes of $\pi$-admissible elements of $\Psi$ to $\mathscr{F}$ given by
\[
    \kappa : [\psi]_\pi \mapsto S^{[\psi]_\pi}, \quad \text{where} \quad S_{i}^{[\psi]_\pi} := \pi(\psi^{-1}),
\]
is a well-defined bijection independent of the choice of $S \in \mathscr{F}$.

We then define a pairing $\{- , -\} : \operatorname{Uch} \times \operatorname{Uch} \to \mathbb{C}$ by
\[
    \{S, S'\} = \frac{(-1)^{\ell(m-1)}}{\tau(m)^\ell}\sum_{\nu \in [\kappa^{-1}(S)]_\pi}\epsilon(\nu)\epsilon(\kappa^{-1}(S')) \prod_{y \in Y} \zeta_m^{-\nu(y)\kappa^{-1}(S')(y)}
\]
if $S,S'$ belong to the same family, where $\epsilon(\psi) = (-1)^{\mathfrak{c}(\psi)}$,
\[
    \mathfrak{c}(\psi) = \#\{(y,y') \in Y \times Y : y < y' \text{ and } \psi(y) < \psi(y')\},
\] 
and $\{S, S'\} = 0$ otherwise.

The restriction $\{-,-\}_{G(m,1,n)}$ of this pairing to the irreducible characters of $G(m,1,n)$ then satisfies (T1), (T2), and (T3) by \cite[Remark 5.3.26 and Corollary 5.3.33]{lasy2012traces}.

\subsection{The Fourier transform for \texorpdfstring{$G(m,m,n)$}{G(m,m,n)} and the exceptional groups}

In \cite[Conjecture 5.4.34]{lasy2012traces}, Lasy conjectures the existence of such a Fourier transform for $G(m,m,n)$ and describes its relation to a ``pre-Fourier'' transform which is a slight modification of the construction in \cite[\S 6C]{malle1995unipotente}. This conjectured transform will satisfy (T1), (T2), and (T3) \cite[Corollary 5.4.38]{lasy2012traces}.

The Fourier transform for the exceptional groups (and for the families $G(m,1,n)$ and $G(m,m,n)$) are contained in the data for GAP3, but the properties of these matrices have not yet appeared in publication. See \cite{malle1998spetses} and \cite{broue2014split}.

\section{Rational Catalan numbers}
\label{sec:ratioCatalan}

For an irreducible complex reflection group $W$ with Coxeter number $h$, define rational $W$-Catalan numbers \cite{gordon2012catalan}
\[
    \operatorname{Cat}_p(W) = \prod_{i=1}^n \frac{p + (pe_i \mod h)}{d_i},
\]
and their $q$-deformations
\[
    \operatorname{Cat}_p(W; q) = \prod_{i=1}^n \frac{[p + (pe_i \mod h)]_q}{[d_i]_q}.
\]

\begin{theorem}
\label{thm:main}
    Let $W$ be an irreducible spetsial complex reflection group with Coxeter number $h$, and let $c$ be a $\zeta_h$-regular element of $W$. Let $\mathbf{c} \in B(W)$ be a lift of $c$ such that $\mathbf{c}^h = \bm{\pi}$. Then
    \[
        \tau_q(T_{\mathbf{c}}^{-p}) = q^{-np}(1-q)^n \operatorname{Cat}_p(W; q).
    \]
\end{theorem}
\begin{proof}
    Starting from \Cref{thm:traceRootOfTwist}, then using \Cref{lem:transform}, \Cref{thm:galDeg}, \Cref{ex:galCox}, and \Cref{ex:galoisFeg}:
    \begin{align*}
        \tau_q(T_{\mathbf{c}}^{-p}) &= \frac{1}{P_W}\sum_{\chi \in \operatorname{Irr}(W)} q^{(h_{\chi} - nh) p/ h} \operatorname{Feg}_\chi (e^{2\pi i p/h}) \operatorname{Deg}_\chi (q) \\
        &= \frac{1}{P_W}\sum_{\chi \in \operatorname{Irr}(W)} q^{(h_{\chi} - nh) p/ h} \operatorname{Feg}_\chi (q) \operatorname{Deg}_\chi (e^{2\pi i p/h}) \\
        &= \frac{1}{P_W} \sum_{k=0}^n (-1)^k q^{(k-n)p} \sum_{i_1 < \cdots < i_k} q^{e_{i_1}(V^{\sigma_p}) + \cdots + e_{i_k}(V^{\sigma_p})} \\
        &= \frac{1}{P_W} q^{-np} \prod_{i=1}^n \left(1 - q^{p + e_i(V^{\sigma_p})}\right) \\
        &= q^{-np}(1-q)^n \prod_{i=1}^n \frac{[p + e_i(V^{\sigma_p})]_q}{[d_i]_q} \\
        &= q^{-np}(1-q)^n \operatorname{Cat}_p(W; q).\qedhere
    \end{align*}
\end{proof}

\begin{remark}
    This proof for $G(m,m,n)$ relies on \Cref{conj:transform}. For the exceptional groups, the result can be established by computer \cite{Code} without assuming the conjecture.
\end{remark}

\begin{remark}
    This result is not true for the groups that are well-generated but not spetsial, and one might reasonably ask where this proof fails. For these groups, it is still true that 
    \[
        \tau_q(T_{\mathbf{c}}^{-p}) = \frac{1}{P_W}\sum_{\chi \in \operatorname{Irr}(W)} q^{(h_{\chi} - nh) p/ h} \operatorname{Feg}_\chi (e^{2\pi i p/h}) \operatorname{Deg}_\chi (q)
    \]
    and
    \[
        q^{-np}(1-q)^n \operatorname{Cat}_p(W; q) = \frac{1}{P_W}\sum_{\chi \in \operatorname{Irr}(W)} q^{(h_{\chi} - nh) p/ h} \operatorname{Feg}_\chi (q) \operatorname{Deg}_\chi (e^{2\pi i p/h}),
    \]
    but the right-hand-sides of these equations are not equal to each other. So there is no Fourier transform for these groups.
\end{remark}

These ``trace techniques'' also allow us to extend the algebraic part of \cite[Corollary 6.15]{galashin2022rational} to spetsial groups. 

\begin{corollary}
\label{cor:parking}
    For $W$ an irreducible spetsial complex reflection group, let $\mathcal{B}$ be a basis of the spetsial Hecke algebra $\mathcal{H}_q(W)$ (adapted to the Wedderburn decomposition), and let $\mathbf{c}$ be a lift of a $\zeta_h$-regular element such that $\mathbf{c}^h = \mathbf{\pi}$. Then
    \[
        \sum_{b \in \mathcal{B}} \tau_q(b^\vee T_{\mathbf{c}^p} b) = (q-1)^n [p]_q^n.
    \]
\end{corollary} 
\begin{proof}
    Using \cite[pg. 226]{geck2000characters},
    \begin{align*}
        \sum_{b \in \mathcal{B}} \tau_q(b^\vee T_{\mathbf{c}^p} b) &= \sum_{b \in \mathcal{B}} \sum_{\chi \in \operatorname{Irr}(W)} \frac{1}{S_{\chi}(q)} \chi_q(b^\vee T_{\mathbf{c}^p} b) = \sum_{\chi \in \operatorname{Irr}(W)} \frac{1}{S_{\chi}(q)} \sum_{b \in \mathcal{B}}\chi_q(b^\vee T_{\mathbf{c}^p}b) \\
        &= \sum_{\chi \in \operatorname{Irr}(W)} (\operatorname{dim}\chi) \cdot \chi_q(T_{\mathbf{c}^p}) = \sum_{\chi \in \operatorname{Irr}(W)} q^{(nh - h_\chi)p/h} \operatorname{Deg}_\chi (1) \operatorname{Feg}_\chi(e^{-2\pi i p / h}) \\
        &= \sum_{k=0}^{n} q^{p(n-k)} (-1)^{k} \binom{n}{k} = (q^p - 1)^n = (q-1)^n [p]_q^n.
    \end{align*}
\end{proof}

This corollary motivates future work:
\begin{question}
    Are there noncrossing parking objects for spetsial complex reflection groups analogous to those in \cite{galashin2022rational} whose enumeration is related to the sum $\sum_{b \in \mathcal{B}} \tau_q(b^\vee T_{\mathbf{c}^p} b)$?
\end{question}

\bibliographystyle{amsalpha}
\bibliography{literature}
\end{document}